\documentclass{amsart}
\usepackage{amsmath, amssymb, amsthm, xcolor}
\usepackage{tikz}
\newtheorem{theorem}{Theorem}
\newtheorem{lemma}{Lemma}
\newtheorem{definition}{Definition}
\newtheorem{property}{Property}
\newtheorem{proposition}{Proposition}
\newtheorem{corollary}{Corollary}
\newtheorem{remark}{Remark}
\newtheorem{example}{Example}
\newcommand{\T}{{C^\infty(S^2(\mathbb{R}^+, \mathbb{R}^+))}}
\newcommand{\ts}{\mathsf{T}}

\newcommand{\D}[1]{{\partial_{\vec{#1}}}}

\begin{document}
\title[Mean field information Gamma calculus on graphs]{Mean field information Hessian matrices on graphs}
\author[Li]{Wuchen Li}
\email{wuchen@mailbox.sc.edu}
\address{Department of Mathematics, University of South Carolina, Columbia, SC 29208.}
\author[Lu]{Linyuan Lu}
\email{lu@mailbox.sc.edu}
\address{Department of Mathematics, University of South Carolina, Columbia, SC 29208.}

\thanks{W. Li thanks a start-up funding from the University of South Carolina. Both W. Li and L. Lu are also supported by NSF RTG: 2038080.}
\maketitle
\begin{abstract}
We derive mean-field information Hessian matrices on finite graphs. The ``information'' refers to entropy functions on the probability simplex. And the ``mean-field" means nonlinear weight functions of probabilities supported on graphs. These two concepts define a mean-field optimal transport type metric. In this metric space, we first derive Hessian matrices of energies on graphs, including linear, interaction energies, entropies. We name their smallest eigenvalues as mean-field Ricci curvature bounds on graphs. We next provide examples on two-point spaces and graph products. We last present several applications of the proposed matrices. E.g., we prove discrete Costa's entropy power inequalities on a two-point space. 
\end{abstract}
\section{Introduction}
Convexities of entropy functions play essential roles in differential geometry, probability, and information theory \cite{BE,Csiszar,vil2008}. It finds vast applications, such as studying or designing fast Markov-Chain-Monte-Carlo (MCMC) algorithms in Bayesian sampling and AI (Artificial Intelligence) inference problems \cite{Duncan}. 

The convexity of entropies is widely studied in probability space embedded with optimal transport metrics. It is useful in establishing information-theoretical inequalities, such as log-Sobolev \cite{Gross,WZ}, Poincar{\'e}, transport-information \cite{OV} and Costa’s entropy power inequalities \cite{Costa, vil2000}. The convexity depends on the Hessian operators of entropy, which forms a generalized Bakry-Emery Gamma calculus \cite{BE,BGL,BF}; see \cite{AGS, vil2008, Lott_Villani, strum} and \cite{LiHess, LiG, LiG1}. However, the classical Gamma calculus requires the sample space to be a continuous space, allowing high order calculus (integration by parts). This property is often missed in a discrete sample space. Recently, a class of {\em discrete optimal transport metrics} have been introduced in \cite{chow2012, EM1, M}. One can apply them in defining ``discrete Gamma calculus'' and Ricci curvature on graphs \cite{EM1, M}. Moreover, the Hessian operators can provide formalisms in establishing convergence rates of discrete-state Markov processes. However, Hessian operators of general energies w.r.t. mean-field optimal transport metrics on graphs are not clear \cite{GC}. 

In this paper, we study mean-field information Hessian matrices on a finite simple graph. And the Hessian matrices are formulated for general energies. Examples include linear, interaction energies, and entropies; see Theorem \ref{t2}. Using the spectral graph theory, we study some explicit bounds for the smallest eigenvalue of Hessian matrices, namely the ``mean-field Ricci curvature lower bound''. Furthermore, we present analytical lower bounds of mean-field Ricci curvature for two-point spaces and graph products. In applications, we demonstrate entropy dissipation properties and mean-field ``log-Sobolev'' inequalities on graphs. We also prove a Costa's entropy power inequality on a two-point graph. 

In literature, there are joint works on discrete Ricci curvatures \cite{Maas2012, EMR4, Jost2, Jost, LinYau2, LinYau, M, Ollivier_Ricci, WZ}; see many references therein. Technically speaking, our methods are closely related to \cite{LinYau2,LinYau} and \cite{Maas2012}. Compared to \cite{LinYau2,LinYau}, we consider a mean-field class of Gamma calculus on graphs, which depends on the functions of discrete probabilities (mean-field weight functions). See Theorem \ref{t2}. Meanwhile, compared to \cite{Maas2012}, we inherit and extend the Gamma calculus defined in \cite{Maas2012}. Firstly, we formulate Hessian matrices of general energies in probability simplex. Secondly, we define a ``transport information mean function'' based on the constant eigenvalue of Hessian operators. Lastly, we formulate analytical bounds for these Hessian matrices on graph products. In particular, we extend the tensor product property in \cite{Maas2012}, which works for the combination of Shannon entropy and logarithm mean function. In Corollary \ref{col3}, we demonstrate that this property works for general energy functions and weight functions. Besides the above comparisons, we apply the proposed Hessian operators to establish Costa's entropy power inequalities on graphs. We expect that our calculation will be useful in establishing analytical bounds for discrete information theory inequalities with applications in machine learning probability models, such as Boltzmann machine; see \cite{LiR}. 

We organize this paper below. In section \ref{section2}, we present the main result. We derive the Hessian matrix of energy functions w.r.t. mean-field optimal transport metrics on a graph. We define a mean-field Ricci curvature lower bound by the smallest eigenvalue of the proposed Hessian matrix. In section \ref{section3} and \ref{section4},  we derive analytical bounds for the Hessian matrices on a two-point space and a graph product. Finally, in section \ref{section6}, we present some applications of the proposed Hessian matrices, such as proving Costa's entropy power's inequalities on graphs. 

\section{Notations and main results}\label{section2}
In this section, we first present all notations, such as  mean-field-optimal-transport metric spaces on graphs. 
See their motivations in appendix. We next formulate the main result, which is the Hessian operators of general energies in the above metric space. Several examples of Hessian operators of energies, such as entropies, linear and interaction energies, are presented. 

\subsection{Notations} The logarithm $\log (x)$ is natural logarithm with base $e$. Let $\mathbb{R}^+$ denote the interval $[0,\infty)$ and 
$\mathbb{R}^{++}$ denote the open interval $(0,\infty)$.
Let $\T$ denote the set of function $\theta\colon \mathbb{R}^{+} \times \mathbb{R}^{+} \to \mathbb{R}^{+}$, such that
\begin{enumerate}
    \item (Regularity): $\theta$ is continuous on $\mathbb{R}^+
\times \mathbb{R}^+ $ and $C^\infty$ on
$(0,\infty) \times(0,\infty)$;
\item (Symmetry): 
$\theta(s,t)=\theta(t,s)$ for $s, t \geq 0$;
\item (Positivity): $\theta(s, t) > 0$ for $s, t > 0$.
\end{enumerate}

Let $G=(V,E)$ be a simple graph with vertex set $V$ and edge set $E$.
Without loss of generality, we often set $V=[n]=\{1,2,\ldots,n\}$, where 
$n$ is the number of vertices. 

A probability distribution on $V$ is 
a vector $p=(p_1, p_2,\ldots, p_n)$ with $p_i\geq 0$ for all $i$ and $\sum_{i=1}^n p_i=1$. The set of all probability distributions forms the standard simplex 
\begin{equation*}
    M=\{(p_i)_{i=1}^n\colon \sum_{i=1}^np_i=1,\quad p_i\geq 0\}\subset \mathbb{R}^n.
\end{equation*}
We can view $M$ as a manifold of dimension $n-1$ with boundary. 
The tangent bundle  $TM$ has a global trivialization with basis 
$e_i= \frac{\partial }{\partial p_i}-\frac{\partial }{\partial p_{i+1}}$
for $1\leq i \leq n-1$. Let $e^*_i$ ($1\leq i\leq n-1)$ be the dual basis in the cotangent bundle $T^*M$. From now on, we focus on the interior of the probability simplex $M$; see related studies on its boundary set in \cite{GLM}. 

For a simple graph $G$, we choose a function $\theta\in \T$ and associate
each edge $ij$ (and a point $p\in M$) with the expression $\theta_{ij}=\theta(p_i,p_j)$.
For a weighted graph $G$, 
we associate each edge $ij$ with the expression $\theta_{ij}=\theta_{ij}(p_i,p_j)$ for some function $\theta_{ij}\in \T$
where the choice of the function $\theta_{ij}$ depending on the edge weight
$w_{ij}$. 

It is convenient to extend to all pair of vertices by setting $\theta_{ij}=0$ for all non-edges $ij$ and $\theta_{ii}=0$ on the diagonal.
The collection $\{\theta_{ij}\}_{ij\in E(G)}$ define a inner product on
$T^*_p\mathbb{R}^n$ as 
$$\langle x,y\rangle =\sum_{ij\in E(G)}\theta_{ij}(x_i-x_j)(y_i-y_j),$$
for any $x=(x_1,\ldots, x_n)$ and $y=(y_1,\ldots, y_n)$.
This inner produce induces an inner product $g$ on
the cotangent space $T_p^*M$. Let $g^{ij}_p=g_p(e^*_i,e^*_j)$.
Then we have the following simple expression:
$$g^{ij}_p= \theta_{ij}-\theta_{(i+1)j}
- \theta_{i(j+1)}+ \theta_{(i+1)(j+1)}.$$

This metric $g$ turns $M$ into a Riemannian manifold. In this paper, we are not interested in the geometry associated to the Levi-Civita connection induced by this Riemannian metric $g$; rather than a non-standard connection $\nabla$.
\begin{definition}
Consider the geodesic equation as 
\begin{equation}\label{eq:geodesic}
\left\{
\begin{aligned}
  \frac{d}{dt}{p}_i &= \sum_{j=1}^n (f_i-f_j) \theta_{ij},\\
  \frac{d}{dt}{f}_i &= -\frac{1}{2} \sum_{j=1}^n (f_i-f_j)^2 \frac{\partial \theta_{ij}}{\partial p_i} + h_i.
\end{aligned}\right.
\end{equation}
\end{definition}
Sometimes, we also use the dot notation to represent the derivative respect to time $t$. 
Let $\gamma\colon [0,1]\to M$ given by 
$\gamma(t)=(p_1(t),\ldots, p_n(t))$
be a curve. The tangent vector 
is given by $\dot \gamma=\sum_{i=1}^n \dot p_i\frac{\partial}{\partial p_i}.$ The vector $f=(f_1,\ldots, f_n)$ lives in $T_p^*\mathbb{R}^n$ so that
$i^*(f)$ is the lift of $\dot \gamma$ in $T_p^*M$ under the Riemannian metric $M$. Here $i^*$ is the pullback map of the standard inclusion map
$i\colon M\to \mathbb{R}^n$. The connection $\nabla$ (depending of the choice of $\{h_i\}$) is chosen so that the geodesic equation
$\nabla_{\dot \gamma(t)}\dot \gamma(t)=0$ is simplified to equation \eqref{eq:geodesic}.
\begin{definition}
Denote 
\begin{equation*}
    \Gamma_1(p,f,f)=\sum_{ij=1}^n \theta_{ij}(f_i-f_j)^2.
\end{equation*}
For simplicity of notation, we denote $\Gamma_1(f,f)=\Gamma_1(p,f,f)$. We call geodesic curve $\gamma(t)$ is {\em constant-speed} if $\Gamma_1(f,f)$ is a constant.
\end{definition}
Then we have the following Lemma.

 \begin{lemma}\label{t1}
The geodesic curve $\gamma(t)$ is constant speed if and only if the vector
$h=(h_1,\ldots, h_n)$ is orthogonal to $f=(f_1,\ldots, f_n)$ point-wisely, i.e.,
$$\sum_{ij=1}^n \theta_{ij} (f_i-f_j)(h_i-h_j) =0.$$
\end{lemma}

Before showing Lemma \ref{t1}, we shall prove two identities based on symmetry and anti-symmetry.
\begin{lemma} \label{l3}
Assume for all $ij$, $a_{ij}=a_{ji}$ and $b_{ij}=-b_{ji}$. We have
\begin{align}
\label{eq:symmetry}
    \sum_{i,j=1}^n a_{ij}x_i&=\frac{1}{2}\sum_{i,j=1}^n a_{ij}(x_i+x_j),\\
    \sum_{i,j=1}^n b_{ij}x_i&=\frac{1}{2}\sum_{i,j=1}^n a_{ij}(x_i-x_j).
    \label{eq:asymmetry}
\end{align}
\end{lemma}

\begin{proof}
By switching index $i$ and $j$, we have
\begin{align*}
\sum_{i,j=1}^n a_{ij}x_i&= \sum_{i,j=1}^n a_{ij}x_j.\\
\sum_{i,j=1}^n b_{ij}x_i&= -\sum_{i,j=1}^n b_{ij}x_j.
\end{align*}
Then by taking average, we can derive
Equations \eqref{eq:symmetry} and \eqref{eq:asymmetry}.
\end{proof}

\begin{proof}[Proof of Lemma \ref{t1}]
\begin{align*}
       \frac{d}{dt}\Gamma_1(f,f)&=\frac{d}{dt}\sum_{i,j=1}^n \theta_{ij}(f_i-f_j)^2\\
       &=  \sum_{i,j=1}^n \dot{\theta}_{ij}(f_i-f_j)^2 +\sum_{i,j=1}^n  2\theta_{ij}(f_i-f_j)(\dot{f}_i-\dot{f}_j)\\
       &=   \sum_{i,j=1}^n  \left(\frac{\partial \theta_{ij}}{\partial p_i} \dot{p}_i +  \frac{\partial \theta_{ij}}{\partial p_j} \dot{p}_j\right)(f_i-f_j)^2 + \sum_{i,j=1}^n  2\theta_{ij}(f_i-f_j)(\dot{f}_i-\dot{f}_j)\\
        &= 2  \sum_{i,j=1}^n \frac{\partial \theta_{ij}}{\partial p_i}\dot{p}_i (f_i-f_j)^2 + 4 \sum_{i,j=1}^n\theta_{ij}(f_i-f_j)\dot{f}_i\\
        &= 2 \sum_{i,j=1}^n \frac{\partial \theta_{ij}}{\partial p_i}
        \sum_{k=1}^n (f_i-f_k) \theta_{ik}
        (f_i-f_j)^2 + 4 \sum_{i,j=1}^n\theta_{ij}(f_i-f_j)\dot{f}_i\\
         &= 2 \sum_{i,j,k=1}^n \frac{\partial \theta_{ik}}{\partial p_i}
        (f_i-f_j) \theta_{ij}
        (f_i-f_k)^2 + 4 \sum_{i,j=1}^n\theta_{ij}(f_i-f_j)\dot{f}_i\\
        &= 4\sum_{i,j=1}^n (f_i-f_j)\theta_{ij}\left[\dot f_i+ \frac{1}{2} \sum_{k=1}^n (f_i-f_k)^2\frac{\partial \theta_{ik}}{\partial p_i}\right]\\
        &=4\sum_{i,j=1}^n (f_i-f_j)\theta_{ij} h_i\\
        &=4\sum_{i,j=1}^n (f_i-f_j)\theta_{ij} (h_i-h_j).
         \end{align*}
Thus, $\Gamma_1(f,f)$ is a constant if and only if  $\sum_{i,j=1}^n (f_i-f_j)\theta_{ij} (h_i-h_j)=0$.    
\end{proof}
\subsection{Hessian operators of energies on graphs}
We now fix any energy function $E\colon M\to \mathbb{R}^n$ and define the Hessian operator of $E$ on $(M, g)$ by 
\begin{equation*}
 \mathrm{Hess}^*_gE(p)(f,f):=\Gamma_2(p,f,f):=\frac{d^2}{dt^2} E(p(t)),
\end{equation*}
where $p(t)$ satisfies the geodesic equation \eqref{eq:geodesic}. Sometimes, we also denote $\Gamma_2(f,f)=\Gamma_2(p,f,f)$. We have the following theorem.
\begin{theorem}[Mean-field information matrices on graphs]\label{t2}
For any energy function $E(p)$, let 
\begin{equation*}
\eta_{ij}=\theta_{ij}\left(\frac{\partial E}{\partial p_i}- \frac{\partial E}{\partial p_i}\right).    
\end{equation*}
Assume that $\nabla E$ is orthogonal to 
$h$, i.e., $\sum_{ij=1}^n \eta_{ij}(h_i-h_j)=0$.
Then we have
\begin{align}
    \Gamma_2(p,f,f)&= \frac{1}{2} \sum_{i,j,k=1}^n (f_i-f_j)^2 
\frac{\partial \theta_{ij}}{\partial p_i} \eta_{ki}
+  \sum_{i,j,k=1}^n (f_i-f_j)(f_i-f_k)
    \frac{\partial \eta_{ij}}{\partial p_i} \theta_{ki} 
    \label{eq:Gamma2_1}
    \\
&=
\sum_{i,j,k=1}^n (f_i-f_j)(f_i-f_k)\left(\frac{1}{2} 
    \frac{\partial \theta_{ij}}{\partial p_i} \eta_{ki} + \frac{1}{2} 
    \frac{\partial \theta_{ki}}{\partial p_k} \eta_{jk} +  
    \frac{\partial \eta_{ij}}{\partial p_i} \theta_{ki}\right) 
    \label{eq:Gamma2_2}
    \\
&=  \frac{1}{2} \sum_{i,j,k=1}^n (f_i-f_j)^2 \left(  
\frac{\partial \theta_{ij}}{\partial p_i} \eta_{ki}
+ \frac{\partial \eta_{ij}}{\partial p_i} \theta_{ki} 
+ \frac{\partial \eta_{jk}}{\partial p_j} \theta_{ij}
-\frac{\partial \eta_{ki}}{\partial p_k} \theta_{jk}
    \right). \label{eq:Gamma2_3}
\end{align}
\end{theorem}
The following lemma is of independent interest. It serves a bridge between spectral graph theory and geometric calculations in probability simplex. 
\begin{lemma}
For any 3-tensor $\{a_{ijk}\}$, $\{b_{ijk}\}$, and any vector
$x=(x_1,\ldots, x_n)$,
we have
\begin{align}
\label{eq:ijk2ij}
\sum_{i,j,k=1}^n a_{ijk}(x_i-x_j)(x_i-x_k)&=\frac{1}{2}\sum_{i,j,k=1}^n(a_{ijk}+a_{jki}-a_{kij})(x_i-x_j)^2.\\
\sum_{i,j,k=1}^n b_{ijk}(x_i-x_j)^2 &= \sum_{i,j,k=1}^n (b_{ijk}+b_{kij})(x_i-x_j)(x_i-x_k). 
\label{eq:ij2ijk}
\end{align}
\end{lemma}
\begin{proof}
Let us prove equation \eqref{eq:ij2ijk} first.
\begin{align*}
    \sum_{i,j,k=1}^n b_{ijk}(x_i-x_j)(x_j-x_k) &= \sum_{i,j,k=1}^n b_{ijk}(x_i-x_j)(x_i-x_j+x_j-x_k)\\
    &= \sum_{i,j,k=1}^n b_{ijk}(x_i-x_j)^2 - \sum_{i,j,k=1}^n b_{ijk}(x_j-x_i)(x_j-x_k)\\
    &=\sum_{i,j,k=1}^n b_{ijk}(x_i-x_j)^2 - \sum_{i,j,k=1}^n b_{kij}(x_i-x_j)(x_i-x_k).
\end{align*}
We now derive equation \eqref{eq:ijk2ij} from equation \eqref{eq:ij2ijk} by setting $b_{ijk}= \frac{1}{2}(a_{ijk}+a_{jki}-a_{kij})$. Observe that
$$b_{ijk}+b_{kij}= \frac{1}{2}(a_{ijk}+a_{jki}-a_{kij}) + \frac{1}{2}(a_{kij}+a_{ijk}-a_{jki})
=a_{ijk}.$$
\end{proof}

\begin{proof}[Proof of Theorem \ref{t2}]
Let $\partial_i E=\frac{\partial E}{\partial p_i}$. Then
we have
\begin{align*}
    \frac{d}{dt} E(p)  
    &= \sum_{i=1}^n \partial_i E \dot p_i\\
    &= \sum_{i=1}^n \sum_{j=1}^n (f_i-f_j) \theta_{ij} \partial_i E\\
    &= \frac{1}{2}\sum_{i,j=1}^n (f_i-f_j) \theta_{ij}
    (\partial_i E-\partial_j E)
    \hspace*{4mm} &\mbox{by \eqref{eq:asymmetry}}\\
    &= \frac{1}{2} \sum_{i,j=1}^n (f_i-f_j) \eta_{ij} \\
    &= \sum_{i,j=1}^n f_i \eta_{ij}. \hspace*{4mm} &\mbox{by \eqref{eq:asymmetry}}
\end{align*}
Thus, we have
 $$ \frac{d^2}{dt^2} E(p) =  \sum_{i,j=1}^n \dot f_i \eta_{ij} + \sum_{i,j=1}^n f_i \dot \eta_{ij}.$$
By plugging in the formula for $\dot p_i$, the first item is
\begin{align*}
    \sum_{i,j,k=1}^n \dot f_i \eta_{ij}
    &= \frac{1}{2}\sum_{i,j,k=1}^n (f_i-f_j)^2 \frac{\partial \theta_{ij}}{\partial p_i} \eta_{ki}
    + \sum_{i,j=1}^n \eta_{ij}h_i
    \\
     &= \frac{1}{2}\sum_{i,j,k=1}^n (f_i-f_j)^2 \frac{\partial \theta_{ij}}{\partial p_i} \eta_{ki}
    + \frac{1}{2}\sum_{i,j=1}^n \eta_{ij}(h_i-h_j)
    \\
    &=\frac{1}{2}\sum_{i,j,k=1}^n (f_i-f_j)^2 \frac{\partial \theta_{ij}}{\partial p_i} \eta_{ki}\\
    &= \frac{1}{2}\sum_{i,j,k=1}^n (f_i-f_j)(f_i-f_k)\left(
    \frac{\partial \theta_{ij}}{\partial p_i} \eta_{ki} + 
    \frac{\partial \theta_{ki}}{\partial p_k} \eta_{jk}\right). \hspace*{4mm}\mbox{by \eqref{eq:ij2ijk}}
\end{align*}
 Here we use the assumption $\sum_{i,j=1}^n \eta_{ij}(h_i-h_j)=0$. We now compute the second item.
 \begin{align*}
     \sum_{i,j=1}^n  f_i \dot \eta_{ij}
     &=  \frac{1}{2} \sum_{i,j=1}^n (f_i-f_j) \dot \eta_{ij}\\
&= \frac{1}{2} \sum_{i,j=1}^n (f_i-f_j) \left(\frac{\partial \eta_{ij}}{\partial p_i} \dot p_i+ \frac{\partial \eta_{ij}}{\partial p_j} \dot p_j \right)\\
       &=  \sum_{i,j=1}^n (f_i-f_j) \frac{\partial \eta_{ij}}{\partial p_i} \dot p_i \hspace*{4mm}&\mbox{by \eqref{eq:symmetry}}\\
       &= \sum_{i,j=1}^n (f_i-f_j) \frac{\partial \eta_{ij}}{\partial p_i} \sum_{k=1}^n (f_i-f_k)\theta_{ik}\\
  &=\sum_{i,j,k=1}^n (f_i-f_j) (f_i-f_k)\frac{\partial \eta_{ij}}{\partial p_i}\theta_{ik}\\
  &=\frac{1}{2} \sum_{i,j,k=1}^n (f_i-f_j)^2 
  \left( \frac{\partial \eta_{ij}}{\partial p_i} \theta_{ki} 
+ \frac{\partial \eta_{jk}}{\partial p_j} \theta_{ij}
-\frac{\partial \eta_{ki}}{\partial p_k} \theta_{jk}
    \right).
  \hspace*{4mm} &\mbox{by \eqref{eq:ijk2ij}}
\end{align*} 
 Combining two items together, we have
\begin{align*}
\Gamma_2(p,f,f)&=
\frac{1}{2} \sum_{i,j,k=1}^n (f_i-f_j)^2 
\frac{\partial \theta_{ij}}{\partial p_i} \eta_{ki}
+  \sum_{i,j,k=1}^n (f_i-f_j)(f_i-f_k)
    \frac{\partial \eta_{ij}}{\partial p_i} \theta_{ki} \\
&=\sum_{i,j,k=1}^n (f_i-f_j)(f_i-f_k)\left(\frac{1}{2} 
    \frac{\partial \theta_{ij}}{\partial p_i} \eta_{ki} + \frac{1}{2} 
    \frac{\partial \theta_{jk}}{\partial p_j} \eta_{ij} +  
    \frac{\partial \eta_{ij}}{\partial p_i} \theta_{ki}\right) \\
&=  \frac{1}{2} \sum_{i,j,k=1}^n (f_i-f_j)^2 \left(  
\frac{\partial \theta_{ij}}{\partial p_i} \eta_{ki}
+ \frac{\partial \eta_{ij}}{\partial p_i} \theta_{ki} 
+ \frac{\partial \eta_{jk}}{\partial p_j} \theta_{ij}
-\frac{\partial \eta_{ki}}{\partial p_k} \theta_{jk}
    \right).
 \end{align*}
 \end{proof}
 From now on, we always assume that the vector $h$ is orthogonal to both $f$ and $\nabla E$. Theorem \ref{t2} implies that
$\Gamma_2(p,f,f)$ can be written as a quadratic form $\sum_{ij=1}^n a_{ij}(f_i-f_j)^2$, where 
\begin{equation*}
\begin{split}
a_{ij}=&\frac{1}{2}\sum_{k=1}^n \Big(  
\frac{\partial \theta_{ij}}{\partial p_i} \eta_{ki}
+ \frac{\partial \eta_{ij}}{\partial p_i} \theta_{ki} 
+ \frac{\partial \eta_{jk}}{\partial p_j} \theta_{ij}
-\frac{\partial \eta_{ki}}{\partial p_k} \theta_{jk}\\
&\qquad-\frac{\partial \theta_{ij}}{\partial p_j} \eta_{jk}
- \frac{\partial \eta_{ij}}{\partial p_j} \theta_{jk} 
-\frac{\partial \eta_{ki}}{\partial p_i} \theta_{ij}
+\frac{\partial \eta_{jk}}{\partial p_k} \theta_{ki}\Big),  
\end{split}
\end{equation*}
is independent of the choice of $\{h_i\}$.  Without loss  of generality, we can set $h_i=0$.

\begin{definition}
   Given a weighted metric function $\{\theta_{ij}\}$ over a graph $G$ and an energy function $E$, the local Ricci curvature bound $\kappa^G(p)$ on graph at a point $p$ is the largest number satisfying
   \begin{equation*}
       \Gamma_2(p,f,f)\geq \kappa^G(p) \Gamma_1(p,f,f),
   \end{equation*}
   for any constant speed geodesics passing through $p$.
\end{definition}
   
 \begin{definition}
   Given a weighted metric $\{\theta_{ij}\}$ over a graph $G$ and an energy function $E$, the global Ricci curvature bound $\kappa^G_0$ on graph is the largest number satisfying 
   \begin{equation*}
       \Gamma_2(p,f,f)\geq \kappa^G_0 \Gamma_1(p,f,f),
   \end{equation*}
   for any constant speed geodesics and any point $p$.
\end{definition}  
\begin{remark}
We remark that the ``Ricci curvature on graph'', by a triplet $(G, \{\theta_{ij}\}, E)$, refers to the smallest eigenvalue of the Hessian matrix of energy function $E$ in $(M, g)$. It is not the Ricci curvature tensor in $(M, g)$. 
\end{remark}

By definition, we have $\kappa^G_0=\min_p\{\kappa^G(p)\}$. When the graph $G$ is clear under context, we will omit $G$ and write $\kappa(p)$ and $\kappa_0$, respectively. From the view of spectral graph theory, we define the Laplacian matrix $L(A)=D-A$, where $D$ is a diagonal matrix of row sum and $A$ is the adjacency matrix of the graph. Let $\Theta$ denote the matrix $(\theta_{ij})$ and $L(\Theta)$ be the Laplacian of $\Theta$. 

\begin{definition}
A pair $(\kappa, \alpha)$ is called an eigenvalue-eigenvector pair of $L(A)$ relative to $L(\Theta)$ if
$$L(A)\alpha=\kappa L(\Theta) \alpha.$$
\end{definition}

Note that $(0, {\bf 1})$ is the trivial eigenvalue-eigenvector pair.
Let $(\kappa_i, \alpha_i)$ (for $1\leq i \leq n-1)$ are all eigenvalue--eigenvector pair sorted in the increasing order of $\kappa_i$.
Each $\kappa_i$ is a function on $M$ while each $\alpha_i$ is a section of $T^*M$. Then we have the following property.
\begin{property}
We have 
\begin{equation*}
   \kappa=\kappa_1. 
\end{equation*}
\end{property}
 \begin{definition}[Constant Hessian operators]
 A triple $(G, \{\theta_{ij}\}, E)$ has a constant Ricci curvature if there is a constant $C$ such that $\Gamma_2(f,f)=C\Gamma_1(f,f)$ for any constant-speed geodesics. 
 \end{definition}
Whenever $(G, \{\theta_{ij}\}, E)$ has a constant Ricci curvature, the Wasserstein distance on graph has a very simple formula. In the next section,
we prove that such such triple exists for $G=K_2$ with any given energy function $E$.

In literature \cite{vil2008}, a known fact is that the Hessian matrix of negative Boltzman-Shannon entropy in Wasserstein-2 metric is the expectation of Gamma two operators; see details in appendix. In this paper, we extend this relation to discrete states for both ``information'' type energies and ``mean-field'' type Wasserstein metrics. For this reason, we name Hessian operators $\Gamma_2(p,f,f)$ {\em mean-field-information Gamma calculus}.  

\subsection{Examples}
We last present several examples of mean-field information Gamma calculus for several well known energy functions.
\begin{corollary}\label{Col1}
The following equalities hold. 
\begin{itemize}
\item[(i)]  Consider a linear energy function:
 \begin{equation*}
     E(p)=\sum_{i=1}^nV_ip_i,
 \end{equation*}
 where $V_i\in\mathbb{R}$, $i=1,\cdots, n$ are given constants. Hence 
\begin{equation*}
\begin{split}
&\Gamma_2(p,f,f)\\
=&\frac{1}{2}\sum_{i,j,k=1}^n(f_i-f_j)^2\Big([\frac{\partial \theta_{ij}}{\partial p_i} \theta_{ki}-\frac{\partial\theta_{jk}}{\partial p_j}\theta_{ij}](V_k-V_j)-\frac{\partial\theta_{ki}}{\partial p_k}\theta_{jk}(V_k-V_i)\Big).
\end{split}
\end{equation*}
\item[(ii)] Consider an interaction energy function:
 \begin{equation*}
    E(p)=\frac{1}{2}\sum_{i=1}^nW_{ij}p_ip_j, 
 \end{equation*}
 where $W_{ij}=W_{ji}\in\mathbb{R}$, $i, j\in\{1,\cdots, n\}$, are given symmetric matrix elements. Hence
\begin{equation*}
\begin{split}
&\Gamma_2(p,f,f)\\
=&\frac{1}{2}\sum_{i,j,k=1}^n(f_i-f_j)^2\Big([\frac{\partial \theta_{ij}}{\partial p_i}\theta_{ki}-\frac{\partial\theta_{jk}}{\partial p_j}\theta_{ij}][(Wp)_k-(Wp)_j]-\frac{\partial\theta_{ki}}{\partial p_k}\theta_{jk}[(Wp)_k-(Wp)_i]\\
&\hspace{3cm}+(W_{ii}-W_{ij})\theta_{ij}\theta_{ki}+(W_{jj}-W_{jk})\theta_{jk}\theta_{ij}-(W_{kk}-W_{ki})\theta_{ki}\theta_{jk}\Big).
\end{split}
\end{equation*}
\item[(iii)]  Consider an entropy function:
 \begin{equation*}
    E(p)=\sum_{i=1}^n U(p_i), 
 \end{equation*}
 where $U\colon \mathbb{R}\rightarrow\mathbb{R}$ is a convex function. Hence
\begin{equation*}
\begin{split}
&\Gamma_2(p,f,f)\\
=&\frac{1}{2}\sum_{i,j,k=1}^n(f_i-f_j)^2\Big([\frac{\partial \theta_{ij}}{\partial p_i}\theta_{ki}-\frac{\partial\theta_{jk}}{\partial p_j}\theta_{ij}][U'(p_k)-U'(p_j)]-\frac{\partial\theta_{ki}}{\partial p_k}\theta_{jk}[U'(p_k)-U'(p_i)]\\
&\hspace{3cm}+U''(p_i)\theta_{ij}\theta_{ki}+U''(p_j)\theta_{jk}\theta_{ij}-U''(p_k)\theta_{ki}\theta_{jk}\Big).
\end{split}
\end{equation*}
\end{itemize}
\end{corollary}
\begin{proof}
The proof follows from Theorem \ref{t2}. 

\noindent(i) Consider $E(p)=\sum_{i=1}^nV_ip_i$. In this case, 
 \begin{equation*}
      \eta_{ij}=\theta_{ij}(V_i-V_j),\qquad \frac{\partial\eta_{ij}}{\partial p_i}=\frac{\partial\theta_{ij}}{\partial p_i}(V_i-V_j).  
 \end{equation*}
Hence
\begin{equation*}
\begin{split}
&\Gamma_2(p,f,f)\\
=&\frac{1}{2}\sum_{i,j,k=1}^n(f_i-f_j)^2\left(  
\frac{\partial \theta_{ij}}{\partial p_i} \eta_{ki}
+ \frac{\partial \eta_{ij}}{\partial p_i} \theta_{ki} 
+ \frac{\partial \eta_{jk}}{\partial p_j} \theta_{ij}
-\frac{\partial \eta_{ki}}{\partial p_k} \theta_{jk}
    \right)\\
=&\frac{1}{2}\sum_{i,j,k=1}^n(f_i-f_j)^2\Big(\frac{\partial \theta_{ij}}{\partial p_i} \theta_{ki}(V_k-V_i)+\frac{\partial\theta_{ij}}{\partial p_i}(V_i-V_j)\theta_{ki}+\frac{\partial\theta_{jk}}{\partial p_j}(V_j-V_k)\theta_{ij}-\frac{\partial\theta_{ki}}{\partial p_k}(V_k-V_i)\theta_{jk}\Big)\\
=&\frac{1}{2}\sum_{i,j,k=1}^n(f_i-f_j)^2\Big([\frac{\partial \theta_{ij}}{\partial p_i} \theta_{ki}-\frac{\partial\theta_{jk}}{\partial p_j}\theta_{ij}](V_k-V_j)-\frac{\partial\theta_{ki}}{\partial p_k}\theta_{jk}(V_k-V_i)\Big).
\end{split}
\end{equation*}

\noindent(ii) Consider $E(p)=\frac{1}{2}\sum_{i=1}^nW_{ij}p_ip_j$. In this case, denote $(Wp)_i=\sum_{j=1}^nW_{ij}p_j$, then
 \begin{equation*}
      \eta_{ij}=\theta_{ij}[(Wp)_i-(Wp)_j],  
 \end{equation*}
 and 
  \begin{equation*}
      \frac{\partial\eta_{ij}}{\partial p_i}=\frac{\partial\theta_{ij}}{\partial p_i}[(Wp)_i-(Wp)_j]+\theta_{ij}(W_{ii}-W_{ij}).
 \end{equation*}
Hence
\begin{equation*}
\begin{split}
&\Gamma_2(p,f,f)\\
=&\frac{1}{2}\sum_{i,j,k=1}^n(f_i-f_j)^2\left(  
\frac{\partial \theta_{ij}}{\partial p_i} \eta_{ki}
+ \frac{\partial \eta_{ij}}{\partial p_i} \theta_{ki} 
+ \frac{\partial \eta_{jk}}{\partial p_j} \theta_{ij}
-\frac{\partial \eta_{ki}}{\partial p_k} \theta_{jk}
    \right)\\
=&\frac{1}{2}\sum_{i,j,k=1}^n(f_i-f_j)^2\Big(\frac{\partial \theta_{ij}}{\partial p_i}\theta_{ki}[(Wp)_k-(Wp)_i]+[\frac{\partial\theta_{ij}}{\partial p_i}[(Wp)_i-(Wp)_j]+\theta_{ij}(W_{ii}-W_{ij})]\theta_{ki}\\
&\hspace{3cm}+\frac{\partial\theta_{jk}}{\partial p_j}[(Wp)_j-(Wp)_k]\theta_{ij}+\theta_{jk}(W_{jj}-W_{jk})\theta_{ij}\\
&\hspace{3cm}-\frac{\partial\theta_{ki}}{\partial p_k}[(Wp)_k-(Wp)_i]\theta_{jk}-\theta_{ki}(W_{kk}-W_{ki})\theta_{jk}\Big)\\
=&\frac{1}{2}\sum_{i,j,k=1}^n(f_i-f_j)^2\Big([\frac{\partial \theta_{ij}}{\partial p_i}\theta_{ki}-\frac{\partial\theta_{jk}}{\partial p_j}\theta_{ij}][(Wp)_k-(Wp)_j]-\frac{\partial\theta_{ki}}{\partial p_k}\theta_{jk}[(Wp)_k-(Wp)_i]\\
&\hspace{3cm}+(W_{ii}-W_{ij})\theta_{ij}\theta_{ki}+(W_{jj}-W_{jk})\theta_{jk}\theta_{ij}-(W_{kk}-W_{ki})\theta_{ki}\theta_{jk}\Big).
\end{split}
\end{equation*}
(iii) Consider $E(p)=\sum_{i=1}^n U(p_i)$. In this case, 
 \begin{equation*}
      \eta_{ij}=\theta_{ij}[U'(p_i)-U'(p_j)],\qquad \frac{\partial\eta_{ij}}{\partial p_i}= \frac{\partial\theta_{ij}}{\partial p_i}[U'(p_i)-U'(p_j)]+\theta_{ij}U''(p_i).   
 \end{equation*}
Hence
\begin{equation*}
\begin{split}
&\Gamma_2(p,f,f)\\
=&\frac{1}{2}\sum_{i,j,k=1}^n(f_i-f_j)^2\left(  
\frac{\partial \theta_{ij}}{\partial p_i} \eta_{ki}
+ \frac{\partial \eta_{ij}}{\partial p_i} \theta_{ki} 
+ \frac{\partial \eta_{jk}}{\partial p_j} \theta_{ij}
-\frac{\partial \eta_{ki}}{\partial p_k} \theta_{jk}
    \right)\\
=&\frac{1}{2}\sum_{i,j,k=1}^n(f_i-f_j)^2\Big(\frac{\partial \theta_{ij}}{\partial p_i}\theta_{ki}[U'(p_k)-U'(p_i)]+[\frac{\partial\theta_{ij}}{\partial p_i}[U'(p_i)-U'(p_j)]+\theta_{ij}U''(p_i)]\theta_{ki}\\
&\hspace{3cm}+\frac{\partial\theta_{jk}}{\partial p_j}[U'(p_j)-U'(p_k)]\theta_{ij}+\theta_{jk}U''(p_j)\theta_{ij}\\
&\hspace{3cm}-\frac{\partial\theta_{ki}}{\partial p_k}[U'(p_k)-U'(p_i)]\theta_{jk}-\theta_{ki}U''(p_k)\theta_{jk}\Big)\\
=&\frac{1}{2}\sum_{i,j,k=1}^n(f_i-f_j)^2\Big([\frac{\partial \theta_{ij}}{\partial p_i}\theta_{ki}-\frac{\partial\theta_{jk}}{\partial p_j}\theta_{ij}][U'(p_k)-U'(p_j)]-\frac{\partial\theta_{ki}}{\partial p_k}\theta_{jk}[U'(p_k)-U'(p_i)]\\
&\hspace{3cm}+U''(p_i)\theta_{ij}\theta_{ki}+U''(p_j)\theta_{jk}\theta_{ij}-U''(p_k)\theta_{ki}\theta_{jk}\Big).
\end{split}
\end{equation*}
\end{proof}

We note that the mean-field information Gamma two operators, a.k.a. Hessian matrices in $(M,g)$, depend on the choices of triplet $(E, \theta, G)$. We next present several examples of $\theta$ and $E$, for which the Hessian matrices have simpler formulations. We remark that particular choices of $\theta$ and $E$ have been widely used in studying Markov processes on discrete states; see details in \cite{BT, Maas2012}.
\begin{example}
For linear energies, the Gamma two operator can be a homogeneous degree one function of $p$, when we select function $\theta$ as a homogeneous degree one function of $p$. E.g., consider $E(p)=\sum_{i=1}^n V_ip_i$ and $\theta_{ij}=\frac{p_i+p_j}{2}$. Then
\begin{equation*}
\begin{split}
\Gamma_2(p,f,f)
=&\frac{1}{8}\sum_{i,j,k=1}^n(f_i-f_j)^2\Big((p_k-p_j)(V_k-V_j)-(p_k+p_j)(V_k-V_i)\Big).
\end{split}
\end{equation*}
\end{example}

\begin{example}
For interaction energies, the Gamma two operator can be a linear function of $p$, when we select function $\theta$ as a constant function of $p$. E.g., consider $E(p)=\frac{1}{2}\sum_{i,j=1}^n W_{ij}p_ip_j$ and $\theta_{ij}=c_{ij}$, where $c_{ij}$ are given constants for any $i,j=1,\cdots, n$. Then 
\begin{equation*}
\begin{split}
&\Gamma_2(p,f,f)\\=&\frac{1}{2}\sum_{i,j,k=1}^n(f_i-f_j)^2\Big((W_{ii}-W_{ij})c_{ij}c_{ki}+(W_{jj}-W_{jk})c_{jk}c_{ij}-(W_{kk}-W_{ki})c_{ki}c_{jk}\Big). 
\end{split}
\end{equation*}
\end{example}
\begin{example}
For entropy energies, the Gamma two operator can be simpler if we select $\theta_{ij}=\frac{p_i-p_j}{U'(p_i)-U'(p_j)}$. In this case, $\eta_{ij}=p_i-p_j$ is a linear function of $p$. Then
\begin{align*}
\Gamma_2(p,f,f)&=  \frac{1}{2} \sum_{i,j,k=1}^n (f_i-f_j)^2 \left(  
\frac{\partial \theta_{ij}}{\partial p_i} (p_k-p_i)
+  \theta_{ki} + \theta_{ij}-\theta_{jk}\right).
 \end{align*}

\end{example}

In next sections, we shall focus on the effect of graph structures in these Hessian matrices, and provide the estimations for their smallest eigenvalues.     
\section{Two point space and effectiveness}\label{section3}
In this section, we consider a complete graph $K_2$ on two vertices. We prove results of mean-field information matrices on the general weight function $\theta_{12}=\theta(p_1,p_2)$ and energy function $E(p_1,p_2)$. 

\subsection{Gamma calculus on a two point space}
We have the following theorem.
\begin{theorem}
On $K_2$, we have
\begin{equation}\label{eq:K2}
    \kappa= \frac{1}{2}\left(
    \frac{\partial \theta_{12}}{\partial p_1}
    -\frac{\partial \theta_{12}}{\partial p_2}
    \right) \left(\frac{\partial E}{\partial p_1}- \frac{\partial E}{\partial p_2}\right)
    + \theta_{12}\left(
    \frac{\partial^2 E}{\partial p_1^2}
    -2\frac{\partial^2 E} {\partial p_1 \partial p_2}
    + \frac{\partial^2 E}{\partial p_2^2}
    \right).
\end{equation}
\end{theorem}

\begin{proof}
Theorem \ref{t2} gives the following formula of $\Gamma_2(f,f)$:
\begin{equation}
    \Gamma_2(f,f)= (f_1-f_2)^2
    \left [-\frac{1}{2} \left(
    \frac{\partial \theta_{12}}{\partial p_1}
    -\frac{\partial \theta_{12}}{\partial p_2}
    \right) \eta_{12}
    + \left(
    \frac{\partial \eta_{12}}{\partial p_1}
    -\frac{\partial \eta_{12}}{\partial p_2}
    \right) \theta_{12}
    \right].
\end{equation}
On $K_2$, both $\Gamma_2(f,f)$ and $\Gamma_1(f,f)$ are scalars. Thus,
we have 
\begin{align*}
\kappa&= \frac{\Gamma_2(f,f)}{\Gamma_1(f,f)}\\
&= -\frac{1}{2} \left(
    \frac{\partial \theta_{12}}{\partial p_1}
    -\frac{\partial \theta_{12}}{\partial p_2}
    \right) \frac{\eta_{12}}{\theta_{12}}
    + \left(
    \frac{\partial \eta_{12}}{\partial p_1}
    -\frac{\partial \eta_{12}}{\partial p_2}\right).
\end{align*}
Plugging in $\eta_{12}=\theta_{12}\left( 
\frac{\partial E}{\partial p_1}
    -\frac{\partial E}{\partial p_2}\right)$ and simplifying it, we get Equation \eqref{eq:K2}.
\end{proof}

Let $\D{12}=\frac{\partial}{\partial p_1}- \frac{\partial}{\partial p_2}$. Then we have the following formula.
\begin{equation}\label{eq:K2in}
    \kappa = \frac{1}{2}(\D{12}\theta_{12}) (\D{12}E) +\theta_{12}(\D{12}^2E).
\end{equation}
Therefore, on $K_2$, scalar $\kappa$ only depends on the values of $\theta_{12}$ and $E$ on the simplex $M=\{(p_1,p_2)\colon p_1+p_2=1, p_1,p_2\geq 0\}$, independent from the values outside $M$.

\subsection{Transport information mean}
In this subsection, we figure out a weight function, which provides us the constant curvature in a two point space. 

Consider the trivial parametric equation of $M$:
$$p_1=x,\quad p_2=1-x, \quad\quad 0\leq x \leq 1.$$
Without causing confusion, we will re-use the notation
$E$ for $E(x,1-x)$ and $\theta$ for $\theta(x,1-x)$.
Equation \eqref{eq:K2in} can be written as
$$\kappa=\frac{1}{2}\frac{d \theta}{dx} \frac{d E}{dx} +\theta\frac{d^2 E}{dx^2}.$$
\begin{proposition}[\cite{Maas2012}]
On $K_2$, the transportation distance between two points $P_1(x_1,1-x_1)$
and $P_2(x_2,1-x_2)$ in $M$ is $$\int_{x_1}^{x_2}\frac{1}{\sqrt{\theta_{12}}} dx.$$
\end{proposition}
\begin{proof}
Note that $\Gamma_1(f,f)=\theta_{12}(f_1-f_2)^2$ is a constant.
With loss of generality, we may assume $\Gamma_1(f,f)=1$ after scaling time.  The geodesic equation on $K_2$ has a very simple form.
$$\frac{dx}{dt}=(f_1-f_2)\theta_{12}=\sqrt{\theta_{12}}.$$
It implies
\begin{equation}
    t=\int\frac{1}{\sqrt{\theta_{12}}} dx+C.
\end{equation}
With $\Gamma_1(f,f)=1$, the geodesic has constant speed 1. Thus
the transportation distance is simply just the difference of
times of two positions. The proof is finished.
\end{proof}


 \begin{theorem} \label{t:const}
 Assume that $E(p_1,p_2)$ is symmetric and concave upward on $M$. 
 Then $\kappa$ is a constant $C$ if and only if
    $$\theta_{12}=2C\frac{E(p_1,p_2)-E(\frac{1}{2}, \frac{1}{2})}
    {(\frac{\partial E}{\partial p_1}-\frac{\partial E}{\partial p_2})^2}$$
    on $M$.
  \end{theorem}

 \begin{proof}
    Let $g:=\theta_{12}(\frac{\partial E}{\partial p_1}-\frac{\partial E}{\partial p_2})^2$. Since $p_1+p_2=1$, we have $\dot p_2=-\dot p_1$. Thus
    \begin{align*}
      \frac{d}{dt} g &= \dot{\theta}_{12}\left(\frac{\partial E}{\partial p_1}-\frac{\partial E}{\partial p_2}\right)^2+ 2 \theta_{12}\left(\frac{\partial E}{\partial p_1}-\frac{\partial E}{\partial p_2}\right)\left(
    \frac{\partial^2 E}{\partial p_1^2}
    -2\frac{\partial^2 E} {\partial p_1 \partial p_2}
    + \frac{\partial^2 E}{\partial p_2^2}
    \right)\dot p_1\\
    &= \left(\frac{\partial \theta_{12}}{\partial p_1}- \frac{\partial \theta_{12}}{\partial p_2}\right)\dot{p}_1\left(\frac{\partial E}{\partial p_1}-\frac{\partial E}{\partial p_2}\right)^2\\
    &\hspace*{4mm} + 2 \theta_{12}\left(\frac{\partial E}{\partial p_1}-\frac{\partial E}{\partial p_2}\right)\left(
    \frac{\partial^2 E}{\partial p_1^2}
    -2\frac{\partial^2 E} {\partial p_1 \partial p_2}
    + \frac{\partial^2 E}{\partial p_2^2}
    \right)\dot p_1\\
    &=2 \dot{p}_1\left(\frac{\partial E}{\partial p_1}- \frac{\partial E}{\partial p_2}\right)
    \left[
    \frac{1}{2}\left(
    \frac{\partial \theta_{12}}{\partial p_1}
    -\frac{\partial \theta_{12}}{\partial p_2}
    \right) \left(\frac{\partial E}{\partial p_1}- \frac{\partial E}{\partial p_2}\right)
    + \theta_{12}\left(
    \frac{\partial^2 E}{\partial p_1^2}
    -2\frac{\partial^2 E} {\partial p_1 \partial p_2}
    + \frac{\partial^2 E}{\partial p_2^2}
    \right)\right]\\
    &=2 \dot{p}_1\left(\frac{\partial E}{\partial p_1}- \frac{\partial E}{\partial p_2}\right) C\\
      &= 2C\frac{d}{dt} E(p).
    \end{align*}
    This implies that $g-2C E(p)=C_1$ for some constant $C_1$. 
    Since $E$ is concave upward and symmetric, it must reach
    minimum value at the middle point $(\frac{1}{2}, \frac{1}{2})$, we have
    $\frac{\partial \theta_{12}}{\partial p_1}(\frac{1}{2}, \frac{1}{2})
    -\frac{\partial \theta_{12}}{\partial p_2}(\frac{1}{2}, \frac{1}{2})=0$.
    We have
    $g(\frac{1}{2}, \frac{1}{2})=0$.
    Thus $C_1=-2C E(\frac{1}{2}, \frac{1}{2}))$.
    Therefore,
    $$\theta_{12}= 2C\frac{E(p_1,p_2)-E(\frac{1}{2}, \frac{1}{2})}{
\left(\frac{\partial E}{\partial p_1}- \frac{\partial E}{\partial p_2}\right)^2}.$$
\end{proof}

One way to extend of the function $\theta$ from $M$ to ${\mathbb{R}^+}^2$ is using 
\begin{center}
(Positive homogeneity): $\theta(\kappa s, \kappa t) = \kappa \theta(s, t)$ for $\kappa > 0$ and $s, t \geq 0$.
\end{center}
Under the assumption of Positive homogeneity, the $theta$ function in Theorem \ref{t:const} can be uniquely extended to
${\mathbb{R}^+}^2$ as
\begin{equation}\label{eq:info_mean}
    \theta(p_1,p_2)=2C(p_1+p_2)\frac{E(\frac{p_1}{p_1+p_2},\frac{p_2}{p_1+p_2})-E(\frac{1}{2}, \frac{1}{2})}
    {(\frac{\partial E}{\partial p_1}-\frac{\partial E}{\partial p_2})^2}.
\end{equation}

\begin{definition}
The function $\theta$ defined by Equation \eqref{eq:info_mean}
is called {\em transport information mean} with respect to $E$.
\end{definition}


The transport information mean provides a very simple transportation distance formula.

\begin{corollary}
When $\theta(p_1,p_2)= 2C\frac{E(p_1,p_2)-E(\frac{1}{2}, \frac{1}{2})}{
\left(\frac{\partial E}{\partial p_1}- \frac{\partial E}{\partial p_2}\right)^2}$, then the transportation distance 
on $K_2$ between $P_1(x_1,1-x_2)$ and  $P_2(x_2,1-x_2)$
is given below:
\begin{equation}
\begin{cases}
      \sqrt{\frac{2}{C}}\left(\sqrt{E(x_1,1-x_1)-E(\frac{1}{2},\frac{1}{2})}
      -\sqrt{E(x_2,1-x_2)-E(\frac{1}{2},\frac{1}{2})}\right) & 
      \mbox{ if } x_1\leq x_2\leq \frac{1}{2};\\
      \sqrt{\frac{2}{C}}\left(\sqrt{E(x_1,1-x_1)-E(\frac{1}{2},\frac{1}{2})}
      +\sqrt{E(x_2,1-x_2)-E(\frac{1}{2},\frac{1}{2})}\right) & 
      \mbox{ if } x_1\leq \frac{1}{2}\leq x_2;\\
      \sqrt{\frac{2}{C}}\left(\sqrt{E(x_2,1-x_2)-E(\frac{1}{2},\frac{1}{2})}
      -\sqrt{E(x_1,1-x_1)-E(\frac{1}{2},\frac{1}{2})}\right) 
      & 
      \mbox{ if } \frac{1}{2}\leq x_1\leq x_2.
    \end{cases}
\end{equation}
\end{corollary}

\begin{definition}
The function $\theta$ is called
{\em distance normalized} if the transportation distance between $(1,0)$ and $(0,1)$ on $K_2$ is $1$, i.e.,
$$\int_0^1 \frac{1}{\sqrt{\theta(x,1-x)}} dx =1.$$
\end{definition}
In particular, 
$\theta(p_1,p_2)= 16[E(0,1)-E(\frac{1}{2}, \frac{1}{2})]\frac{E(p_1,p_2)-E(\frac{1}{2}, \frac{1}{2})}{
\left(\frac{\partial E}{\partial p_1}- \frac{\partial E}{\partial p_2}\right)^2}$ is distance normalized.

\begin{theorem}
Assume that $E(p_1,p_2)$ is symmetric and concave upward on $M$.
For any $\theta$, the global 
Ricci curvature bound $\kappa_{min}$ satisfies $$\kappa_{min}\leq 8\frac{E(1,0)-E(\frac{1}{2},\frac{1}{2})}{
\left(\int_{0}^1 \frac{1}{\sqrt{\theta(x,1-x)}} dx\right)^2}
.$$
The equality holds if and only if $\theta$ is a transport information mean, i.e,
$$\theta(p_1,p_2)=2C \frac{E(p_1,p_2)-E(\frac{1}{2}, \frac{1}{2})}
    {(\frac{\partial E}{\partial p_1}-\frac{\partial E}{\partial p_2})^2},$$
    on $M$ for some constant $C$.
\end{theorem}
\begin{proof}
If $\kappa_{min}\leq 0$, the assertion holds trivially. We may assume $\kappa_{min}>0$.
If $\theta$ is scaled by a constant factor $C$, then both sides of inequalities are scaled by a factor of $C$. Without loss of generality, we may assume $\theta$ is normalized, i.e.,
$$\int_{0}^1 \frac{1}{\sqrt{\theta(x,1-x)}} dx=1.$$
There is a unit-speed geodesic $\gamma(t)$ with $\gamma(0)=(0,1)$ and $\gamma(1)=(1,0)$.
Similar to the proof of Theorem \ref{t:const}, set $g=\theta_{12}(\frac{\partial E}{\partial p_1}-\frac{\partial E}{\partial p_2})^2$. By symmetry of $\theta$ and $E$, we have
$\gamma(\frac{1}{2})=(\frac{1}{2},\frac{1}{2})$. Note that
$E(\gamma(t))$ is increasing on $[\frac{1}{2}, 1]$.
For $t\in [\frac{1}{2}, 1]$, we have
\begin{align*}
g(\gamma(t))-g(\gamma(\frac{1}{2})) &= \int_{\frac{1}{2}}^t
\frac{d g(\gamma(t))}{dt} dt\\
&= \int_{\frac{1}{2}}^t 2\kappa \frac{d E(\gamma(t)}{dt} dt\\
&\geq 2\kappa_{min} \int_{\frac{1}{2}}^t  \frac{d E(\gamma(t)}{dt} dt\\
&= 2\kappa_{min} \left(E(\gamma(t))- E(\gamma(\frac{1}{2}))\right).
\end{align*}
This implies, for any $t\in [\frac{1}{2},1]$, we have
$$\frac{1}{\sqrt{\theta(\gamma(t))}} \leq \frac{1}{\sqrt{2\kappa_{min}}} \frac{\frac{d E(\gamma(t))}{dt}}{\sqrt{E(\gamma(t))- E(\gamma(\frac{1}{2})}}.$$
Integrate both side from $\frac{1}{2}$ to $1$. We get
\begin{align*}
    \frac{1}{2}&=\int_{\frac{1}{2}}^1 \frac{1}{\sqrt{\theta(\gamma(t))}} dt\\
    &\leq \int_{\frac{1}{2}}^1 \frac{1}{\sqrt{2\kappa_{min}}} \frac{\frac{d E(\gamma(t))}{dt}}{\sqrt{E(\gamma(t))- E(\gamma(\frac{1}{2})}}dt\\
    &= \frac{2}{\sqrt{2\kappa_{min}}} \left(\sqrt{E(\gamma(1))- E(\gamma(\frac{1}{2}))}\right).
\end{align*}
This implies
$$\kappa_{min}\leq 8\left(E(1,0)-E(\frac{1}{2},\frac{1}{2})\right).$$
Equality holds if and only if $\kappa$ is a constant. Thus
$\theta$ must be the information transportation mean.
\end{proof}

\begin{definition}
For a fixed symmetric and concave upward energy function $E$, the effectiveness of an $\theta$ function (on $K_2$ relate to $E$) is defined as
$$EFCT(\theta)=
\frac{\kappa_{min}\left(\int_{0}^1 \frac{1}{\sqrt{\theta(x,1-x)}} dx\right)^2}{
8\left(E(1,0)-E(\frac{1}{2},\frac{1}{2})\right)
}.$$
\end{definition}
\subsection{Effectiveness on negative Boltzman-Shannon Entropy}
In this subsection, we choose $E(P)$ be the negative Boltzmann-Shannon entropy function with base $e$:
$$E(p)=-H(p)=p_1\log p_1 + p_2 \log p_2.$$
Then $E(p)$ is symmetric and concave upward with
$E(0,1)=E(1,0)=0$ and $E(\frac{1}{2},\frac{1}{2})=-\ln(2).$ Towards this energy function, we demonstrate the effectiveness of a $\theta$ function. 

\subsubsection{Transport information mean}
Consider 
$$\theta=16\ln (2) \frac{\left(p_i\log p_i+p_j\log p_j - (p_i+p_j)\log \frac{p_i+p_j}{2}\right)}
{(\log p_i -\log p_j)^2}.$$ 
The transportation distance from $(0,1)$ to $(1,0)$ is $1$.
The curvature is $8\ln(2)$ with the effectiveness 100\%.
 
 \subsubsection{Arithmetic mean}
  Consider 
 $$\theta_{ari}=\frac{p_1+p_2}{2}.$$
 Then the local Ricci curvature bound on $K_2$ is
 $$\kappa(p_1,p_2)= \frac{1}{2p_1p_2}.$$
 The transportation distance from $(0,1)$ to $(1,0)$ is $\sqrt{2}$.
 Then, the global Ricci curvature bound on $K_2$ is
 $\kappa_{min}=2.$
 The effectiveness of $\theta_{alg}$ is 
 $$EFCT(\theta_{ari})=\frac{1}{2\ln(2)}\approx 72.13475205.$$
 
 \subsubsection{Geometric mean}
  Consider 
 $$\theta_{geo}=\sqrt{p_1p_2}.$$
 Then the local Ricci curvature bound on $K_2$ is
 $$\kappa(p_1,p_2)= \frac{1}{\sqrt{p_1p_2}}-\frac{(p_1-p_2)(\log p_1 -\log p_2)}{4 \sqrt{p_1p_2}}.$$
 The transportation distance from $(0,1)$ to $(1,0)$ is $\frac{2\Gamma(3/4)^2}{\sqrt{\pi}}\approx 1.694426169$.
 Then, the global Ricci curvature bound on $K_2$ is
 $\kappa_{min}=-\infty.$
 The effectiveness of $\theta_{alg}$ is 
 $$EFCT(\theta_{geo})=-\infty.$$

 \subsubsection{Logarithmic mean}
 Consider 
 $$\theta_{log}=\frac{p_1-p_2}{\log p_1 - \log p_2}.$$
 The transportation distance from $(0,1)$ to $(1,0)$ is
 $1.558707451\ldots$.
 The local Ricci curvature bound on $K_2$ is
 \begin{equation}\label{lmlocal}
      \kappa(p_1,p_2)= 1+ \frac{p_1^2-p_2^2}{2(\log p_1 -\log p_2)p_1p_2}.
 \end{equation}
 Then, the global Ricci curvature bound on $K_2$ is
 $\kappa_{min}=2.$
 The effectiveness of the logarithmic mean is
 $$EFCT(\theta_{log})=\frac{2}{8\ln(2)}\left(\int_0^1\sqrt{\frac{\log x -\log (1-x)}{x-(1-x)}} dx\right)^2\approx 87.62817572\%.$$

  \subsubsection{Classical spectral graph mean}
  Consider 
\begin{equation*}
    \theta_{sg}=\frac{(\sqrt{p_1}-\sqrt{p_2})^2}{(\log p_1-\log p_2)^2}.
\end{equation*}
 Then the local Ricci curvature bound on $K_2$ is
 $$\kappa(p_1,p_2)= \frac{p_1-p_2}{2\sqrt{p_1p_2}(\log p_1 -\log p_2)}.$$
 Then, the global Ricci curvature bound on $K_2$ is
 $\kappa_{min}=\frac{1}{2}.$
 The transportation distance from $(0,1)$ to $(1,0)$ is
 $3.232504051\ldots$.
 The effectiveness of the spectral graph mean is
 $$EFCT(\theta_{sg})=\frac{\frac{1}{2}}{8\ln(2)}\left(\int_0^1
 \frac{\log x -\log (1-x)}{\sqrt{x}-\sqrt{1-x}} dx\right)^2\approx 94.21774637\%.$$

\section{$C_4$-property and graph product}\label{section4}
For any edge $ij$ in a graph $G$, let 
\begin{equation}
    \Gamma^{ij}_2(f,f)= (f_i-f_j)^2
    \left [-\frac{1}{2} \left(
    \frac{\partial \theta_{ij}}{\partial p_i}
    -\frac{\partial \theta_{ij}}{\partial p_j}
    \right) \eta_{ij}
    + \left(
    \frac{\partial \eta_{ij}}{\partial p_i}
    -\frac{\partial \eta_{ij}}{\partial p_j}
    \right) \theta_{ij}
    \right].
\end{equation}

\begin{definition}
For any graph $G$ and energy function $E$,  we call that function $\theta$ has {\em  $G$-property} with respect to $E$ if 
$$\Gamma^G_2(f,f)\geq \sum_{ij\in E(G)}\Gamma^{ij}_2(f,f),$$
 holds for any constant-speed geodesics on $M$.
\end{definition}

Given two graphs $G$ and $H$, the Cartesian product of graph $G\square H$, is a new graph with the vertex set $V(G)\times V(H)$ and the edges sets consisting of all pairs $((u_1, v_1), (u_2,v_2))$ if 
\begin{enumerate}
    \item $u_1u_2\in E(G)$ and $v_1=v_2$.
    \item $u_1=u_2$ and $v_1v_2\in E(H)$.
\end{enumerate}

\begin{theorem}\label{graphproduct}
Suppose that $\theta$ has $C_4$-property with respect to $E$. Then for any two graphs $G$ and $H$,
we have
$$\kappa^{G\square H}_0\geq \min\{\kappa^G_0, \kappa^H_0\}.$$
\end{theorem}
\begin{proof}
Assume $\kappa_0$ is the minimum of $\kappa^G_0$ and $\kappa^H_0$. It is sufficient to show $\kappa^{G\square H}_0\geq \kappa_0$.

For any vertex $v\in V(H)$, let $G\times \{v\}$ be the induced subgraph of $G\square H$ on the vertex set $V(G)\times \{v\}$.
For any vertex $u\in V(G)$, let $\{u\}\times H$ be the induced subgraph of $G\square H$ on the vertex set $\{u\} \times V(H)$.
For any edge $u_1u_2\in E(G)$ and $v_1v_2 \in E(H)$, let $C_4:=u_1u_2\square v_1v_2$ be the induced subgraph of $G\square H$ on the four vertices
$\{(u_i,v_j)\colon i,j=1,2\}$.  We now consider the expression of $\Gamma^{G\square H}_2(f,f)$, all nonzero terms are divided into three groups:
\begin{enumerate}
    \item Three vertices $i,j,k$ are all in some $G\times \{v\}$ for some vertex $v\in V(H)$.
    \item Three vertices $i,j,k$ are all in some $\{u\}\times H$ for some vertex $u\in V(G)$.
    \item Three vertices $i,j,k$ are in $C_4:=u_1u_2\square v_1v_2$ for some edge $u_1u_2\in E(G)$ and $v_1v_2 \in E(H)$. In this case, let $F(u_1,u_2, v_1,v_2)$ denote the difference (in this copy of $C_4$) $$\Gamma^{C_4}(f|_{V(C_4)}, f|_{V(C_4)}) - \sum_{ij\in E(C_4)}\Gamma^{ij}_2(f|_{V(C_4)}, f|_{V(C_4)}).$$
\end{enumerate}

We have
\begin{align*}
    \Gamma^{G\square H}_2(f,f) &= \frac{1}{2} \sum_{i,j,k\in V(G)\times V(H)} (f_i-f_j)^2 
\frac{\partial \theta_{ij}}{\partial p_i} \eta_{ki}
+  \sum_{i,j,k\in V(G)\times V(H)} (f_i-f_j)(f_i-f_k)
    \frac{\partial \eta_{ij}}{\partial p_i} \theta_{ki} \\
&=  \sum_{u\in V(G)} \Gamma^{\{u\}\times H}_2(f|_{\{u\}\times H},f|_{\{u\}\times H}) + \sum_{v\in V(H)} \Gamma^{G\times\{v\}}_2(f|_{G\times\{v\}},f|_{G\times\{v\}})\\
&\hspace*{1cm} + \sum_{u_1u_2\in E(G), v_1v_2\in E(H)} F(u_1,u_2,v_1,v_2)\\
&\geq  \sum_{u\in V(G)} \Gamma^{\{u\}\times H}_2(f|_{\{u\}\times H},f|_{\{u\}\times H}) + \sum_{v\in V(H)} \Gamma^{G\times\{v\}}_2(f|_{G\times\{v\}},f|_{G\times\{v\}})\\
&\geq  \sum_{u\in V(G)} \kappa_0 \Gamma^{\{u\}\times H}_1(f|_{\{u\}\times H},f|_{\{u\}\times H}) + \sum_{v\in V(H)} \kappa_0 \Gamma^{G\times\{v\}}_1(f|_{G\times\{v\}},f|_{G\times\{v\}})\\
&= \kappa_0 \Gamma_1(f,f).
\end{align*}
Therefore, $\kappa^{G\square H}(p)\geq \kappa_0$.
\end{proof}

\begin{theorem}\label{tC4prop1}
Suppose that $\theta$ function is 1-homogenous, convex, and satisfying $\theta(p_i,p_j)=\frac{p_i-p_j}{\frac{\partial E}{\partial p_i}- \frac{\partial E}{\partial p_j}}$.
Then $\theta$ has the $C_4$-property with respect to $E$.
\end{theorem}

The following Lemma was proved by Erbar and Maas.

\begin{lemma}[Erbar-Mass \cite{Maas2012}, Lemma 2.2] \label{Erbar-Mass-Lemma}
Supppose that $\theta$ is positive, homogeneous of degree one, and concave. Then for all $s, t, u, v > 0$, we have
\begin{align}
\label{eq:st}
    s \frac{\partial \theta(s,t)}{\partial s} + t\frac{\partial \theta(s,t)}{\partial t} &=\theta(s,t),\\
   s \frac{\partial \theta(u,v)}{\partial u} + t\frac{\partial \theta(u,v)}{\partial v} &\geq \theta(s,t).
   \label{eq:uv}
\end{align}
\end{lemma}
We now are ready to prove Theorem \ref{tC4prop1}.
\begin{proof}
Since $\theta(p_i,p_j)=\frac{p_i-p_j}{\frac{\partial E}{\partial p_i}- \frac{\partial E}{\partial p_j}}$, we have $\eta(p_i,p_j)=p_i-p_j$. We have
\begin{align*}
 &\qquad  \Gamma^{C_4}_2(f,f)-\sum_{ij\in E(C_4)}\Gamma^{ij}_2(f,f)\\
 &= \frac{1}{2}\sum_{i=1}^4 (f_i-f_{i+1})^2 \left( \frac{\partial \theta_{i,i+1}}{\partial p_i} \eta_{i-1,i}- \frac{\partial \theta_{i,i+1}}{\partial p_{i+1}} \eta_{i+1,i+2}\right) \\
 &+ \sum_{i=1}^4 (f_i-f_{i-1})(f_i-f_{i+1}) \left( \frac{\partial \eta_{i,i+1}}{\partial p_i} \theta_{i,i-1}{+}\frac{\partial \eta_{i,i-1}}{\partial p_{i}} \theta_{i,i+1}\right) \\
 &=\frac{1}{2}\sum_{i=1}^4 (f_i-f_{i+1})^2 \left( \frac{\partial \theta_{i,i+1}}{\partial p_i} \eta_{i-1,i}- \frac{\partial \theta_{i,i+1}}{\partial p_{i+1}} \eta_{i+1,i+2} \right) \\
 &+ \sum_{i=1}^4  (f_i-f_{i-1})(f_i-f_{i+1}) (\theta_{i,i-1}+\theta_{i,i+1})\\
 &= \frac{1}{2}\sum_{i=1}^4 (f_i-f_{i+1})^2 \left( \frac{\partial \theta_{i,i+1}}{\partial p_i} \eta_{i-1,i}- \frac{\partial \theta_{i,i+1}}{\partial p_{i+1}} \eta_{i+1,i+2} +\theta_{i, i+1} -\theta_{i+2, i-1}\right) \\
 &+ \frac{1}{2}(\theta_{12}+\theta_{23}+\theta_{34}+\theta_{41})(f_1-f_2+f_3-f_4)^2.
\end{align*}
In the last step, we apply a straightforward but non-trivial identity:
\begin{align*}
&\hspace*{-1cm}  
\sum_{i=1}^4 (f_i-f_{i+1})^2 (-\theta_{i, i+1} +\theta_{i+2, i-1}) + 2  \sum_{i=1}^4  (f_i-f_{i-1})(f_i-f_{i+1}) (\theta_{i,i-1}+\theta_{i,i+1}) \\
&=(\theta_{12}+\theta_{23}+\theta_{34}+\theta_{41})(f_1-f_2+f_3-f_4)^2.
\end{align*}
It suffices to show the coefficient of each square term is positive.
We apply Lemma \ref{Erbar-Mass-Lemma}. First apply Equation \eqref{eq:st} to $\theta(p_i, p_{i+1})$.
\begin{align*}
  &\frac{\partial \theta_{i,i+1}}{\partial p_i} \eta_{i-1,i}- \frac{\partial \theta_{i,i+1}}{\partial p_{i+1}} \eta_{i+1,i+2} +\theta_{i, i+1} -\theta_{i+2, i-1}\\ 
  =&\frac{\partial \theta_{i,i+1}}{\partial p_i} (p_{i-1} -p_{i})- \frac{\partial \theta_{i,i+1}}{\partial p_{i+1}} (p_{i+1}-p_{i+2}) + \left(p_i \frac{\partial \theta_{i,i+1}}{\partial p_i}+ p_{i+1}\frac{\partial \theta_{i,i+1}}{\partial p_{i+1}}\right)  -\theta_{i+2, i-1}\\
  =&p_{i-1} \frac{\partial \theta_{i,i+1}}{\partial p_i}+ p_{i+2}\frac{\partial \theta_{i,i+1}}{\partial p_{i+1}} -\theta_{i-1, i+2}\\
\geq& 0.
\end{align*}
The last step applies inequality \eqref{eq:uv}.
\end{proof}

\begin{corollary}\label{col3}
Suppose that $\theta$ function is 1-homogenous, convex, and satisfying $\theta(p_i,p_j)=\frac{p_i-p_j}{\frac{\partial E}{\partial p_i}- \frac{\partial E}{\partial p_j}}$. Then
$$\kappa^{Q_n}_0\geq \kappa^{K_2}_0. $$
\end{corollary}




\section{Applications of Hessian matrices on graphs}\label{section6}
In this section, we apply the mean--field information Gamma calculus to study dynamics on graphs. We first study convergence behaviors of several dissipative dynamics on graphs, including heat equations on graphs. We then formulate several related functional inequalities of general energies on graphs. We last prove the analog of Costa's power inequality on a two point graph.  
\subsection{Entropy dissipation on graphs} 
We first prove the convergence property of discrete heat type equations. Consider a convex energy function as $E(p)$. Denote its minimizer in probability simplex as 
\begin{equation*}
\pi=\arg\min_p\Big\{E(p):\sum_{i=1}^np_i=1,~p_i\geq 0\Big\}.  
\end{equation*}
Consider the following initial value dynamics: 
\begin{equation}\label{dynamics}
\frac{dp_i}{dt}=-\Big(L(\Theta)\nabla_p E(p)\Big)_i=\sum_{j=1}^n(\frac{\partial}{\partial p_j}E(p)-\frac{\partial}{\partial p_i}E(p))\theta_{ij}(p).     
\end{equation}
We notice that equation \eqref{dynamics} is a generalization of heat flows on graphs. In other words, if we select $\theta_{ij}=\frac{p_i-p_j}{\frac{\partial E}{\partial p_i}-\frac{\partial E}{\partial p_j}}$ and assume $\theta_{ij}\geq 0$ for any $p\in M$, then the equation \eqref{dynamics} forms a discrete heat equation:
\begin{equation*}
\frac{dp_i}{dt}=\sum_{ij\in E(G)}(p_j-p_i). 
\end{equation*}

We next demonstrate that function $E$ is a Lyapunov function for dynamics \eqref{dynamics}.  
\begin{lemma}[First and second order De-Bruijn equalities on graphs]\label{lem5}
Suppose $p(t)$ satisfies equation \eqref{dynamics}, then the first order time derivative of $E$ follows 
\begin{equation*}
    \frac{d}{dt}E(p(t))=-\mathrm{I}(p(t)),
\end{equation*}
where $I\colon M\rightarrow\mathbb{R}$ is a ``mean-field Fisher information functional" defined as
\begin{equation*}
\begin{split}
\mathrm{I}(p):=&\Gamma_1(p,\nabla_p E(p), \nabla_p E(p))\\
=&\sum_{i,j=1}^n(\frac{\partial}{\partial p_i} E(p)-\frac{\partial}{\partial p_j}E(p))^2\theta_{ij}(p).     
\end{split}
\end{equation*}
In addition, the second order time derivative of $E$ satisfies
\begin{equation*}
    \frac{d^2}{dt^2}E(p(t))=-\frac{d}{dt}\mathrm{I}(p(t))=2\mathrm{J}(p(t)),
\end{equation*}
where $J\colon M\rightarrow\mathbb{R}$ is a functional defined as 
\begin{equation*}
\begin{split}
\mathrm{J}(p):=&\Gamma_2(p,\nabla_pE(p), \nabla_pE(p))\\
=&\frac{1}{2} \sum_{i,j,k=1}^n (\frac{\partial}{\partial p_i}E(p)-\frac{\partial}{\partial p_j}E(p))^2 \left(  
\frac{\partial \theta_{ij}}{\partial p_i} \eta_{ki}
+ \frac{\partial \eta_{ij}}{\partial p_i} \theta_{ki} 
+ \frac{\partial \eta_{jk}}{\partial p_j} \theta_{ij}
-\frac{\partial \eta_{ki}}{\partial p_k} \theta_{jk}
    \right).
\end{split}
\end{equation*}
\end{lemma}
\begin{proof}
The proof follows from the definitions of gradient and Hessian operators defined in $(M, g)$. See proofs in appendix subsection \ref{71}. 
\end{proof}
We are ready to state the convergence behavior of dynamics \eqref{dynamics}, using energy function $E$ as the Lyapunov function. The following results demonstrate that $p(t)$ converges to $\pi$ exponentially fast. And the rate can be characterized by the proposed mean-field Ricci curvature lower bound. 
\begin{corollary}[Entropy dissipation on graphs]
Suppose $\kappa>0$ and $p(t)$ satisfies equation \eqref{dynamics}, then 
\begin{equation*}
E(p(t))-E(\pi)\leq e^{-2\kappa t}(E(p_0)-E(\pi)).     
\end{equation*}
\end{corollary}
\begin{proof}
The proof follows from the Gronwall's inequality. See also proofs in appendix's section \ref{71}. We notice that the mean field Ricci curvature is defined as
\begin{equation*}
\Gamma_2(p, f,f)\geq \kappa\Gamma_1(p, f,f),
\end{equation*}
for any $f\in \mathbb{R}^n$. 
This implies the fact that 
\begin{equation*}
\mathrm{J}(p)=\Gamma_2(p, f,f)|_{f=\nabla_pE(p)}\geq \kappa\Gamma_1(p, f,f)|_{f=\nabla_pE(p)}=\kappa \mathrm{I}(p).  
\end{equation*}
From Lemma \ref{lem5}, the above inequality implies that 
\begin{equation*}
\frac{d^2}{dt^2}E(p(t))\geq -2\kappa\frac{d}{dt}E(p(t)).  
\end{equation*}
Integrating in a time domain $[t, \infty)$, we have
\begin{equation*}
\frac{d}{dt}(E(p(t))-E(\pi))\leq -2\kappa (E(p(t))-E(\pi)).    
\end{equation*}
Following the Grownwall's inequaity, we prove the result. 
\end{proof}
\begin{remark}
The above result could also provide a convergence rate for the discrete heat equation, in term of general Lyapunov function $E(p)$. E.g, if $E(p)=\sum_{i=1}^n p_i\log p_i$ and $\theta_{ij}=\frac{p_i-p_j}{\log p_i-\log p_j}$, it recovers the ones derived in \cite{Maas2012,M} and \cite{BT}. 
\end{remark}
\begin{remark}
We remark that the optimal rate of convergence is given from
\begin{equation*}
\min_{p\in M}\frac{\mathrm{J}(p)}{\mathrm{I}(p)}\geq \kappa.    
\end{equation*}
There could exist energy functions $E$, such that $\kappa$ is not the optimal rate. A direct comparison between $\mathrm{I}$ and $\mathrm{J}$ is called entropy dissipation method; see \cite{WZ} and many references therein. 
\end{remark}

\subsection{Functional inequalities on graphs}
We next present several functional inequalities, which can be derived by the entropy dissipation result on graphs. 
\begin{corollary}[Functional inequalities on graphs]
Suppose $\kappa>0$, then the mean-field Log-Sobolev inequality on a graph holds: 
\begin{equation}\label{lsi}
E(p)-E(\pi)\leq \frac{1}{2\kappa}\sum_{i,j=1}^n(\frac{\partial}{\partial p_i} E(p)-\frac{\partial}{\partial p_j}E(p))^2\theta_{ij}(p),
\end{equation}
for any $p\in M$. 
 \end{corollary}
\begin{proof}
The proof of above inequalities follow the definitions of gradient and Hessian operators in $(M, g)$. See proofs in appendix's section \ref{71}. See also \cite{OV} for other related inequalities.
\end{proof}
Here we present several examples of inequalities on graphs \eqref{lsi}. 
\begin{example}
The following functional inequalities hold. 
\begin{itemize}
\item[(i)] Linear energy: Let $E(p)=\sum_{i=1}^nV_ip_i$. Then \eqref{lsi} satisfies
\begin{equation*}
\sum_{i=1}^n V_ip_i-\sum_{i=1}^nV_i\pi_i\leq \frac{1}{2\kappa}\sum_{i,j=1}^n(V_i-V_j)^2\theta_{ij} (p).   
\end{equation*}
\item[(ii)] Interaction energy: Let $E(p)=\frac{1}{2}\sum_{i=1,j}^nW_{ij}p_ip_j$. Then \eqref{lsi} forms 
\begin{equation*}
\frac{1}{2}\sum_{i=1,j}^nW_{ij}p_ip_j-\frac{1}{2}\sum_{i=1,j}^nW_{ij}\pi_i\pi_j\leq \frac{1}{2\kappa}\sum_{i,j=1}^n([Wp]_i-[Wp]_j)^2\theta_{ij} (p). 
\end{equation*}

\item[(iii)] Entropy: Let $E(p)=\sum_{i=1}^nU(p_i)$. Then \eqref{lsi} forms \begin{equation*}
\sum_{i=1}^nU(p_i)-\sum_{i=1}^nU(\pi_i)\leq \frac{1}{2\kappa}\sum_{i,j=1}^n(U'(p_i)-U'(p_j))^2\theta_{ij}(p). 
\end{equation*}
In particular, if we further choose $\theta$ as the transport information mean, then \eqref{lsi} satisfies
\begin{equation*}
\sum_{i=1}^nU(p_i)-\sum_{i=1}^nU(\pi_i)\leq \frac{1}{2\kappa}\sum_{i,j=1}^n(U(p_i)+U(p_j)-U(\frac{p_i+p_j}{2}))^2. 
\end{equation*}
\end{itemize}\end{example}
\begin{remark}
The proposed Hessian matrix is also useful in proving energy splitting functional inequalities on graphs proposed in \cite{LL}; see its continuous formulation in \cite{BCGL}. We leave their detailed studies in future works. 
\end{remark}
\subsection{Costa’s entropy power inequality on graphs}
We last apply the proposed Hessian matrices to prove discrete Costa's entropy power inequalities. They are discrete analog of the ones in continuous sample space, which is important in information theory \cite{Costa,GSV, LLPS,vil2000}. 

Denote an energy function as 
\begin{equation*}
    N(p)=e^{-\frac{2}{m}E(p)},
\end{equation*}
where $E(p)=\sum_{i=1}^nU(p_i)$ and $U$ is a given convex function. Suppose that there exists a positive constant $m\in\mathbb{R}_+$, which is defined as
\begin{equation*}
\frac{1}{m}:=\min_{p\in M}\frac{\Gamma_2(p,\nabla_pE(p), \nabla_p E(p))}{\Gamma_1(p,\nabla_p E(p), \nabla_pE(p))^2},
\end{equation*}
where we select a weight function as
\begin{equation*}
    \theta_{ij}=\frac{p_i-p_j}{U'(p_i)-U'(p_j)}. 
\end{equation*}
We are ready to prove the main result. 
\begin{theorem}[Discrete Costa’s entropy power inequality]
Consider a discrete heat equation:
\begin{equation}\label{dh}
\frac{d p_i}{dt}=\frac{1}{2}\sum_{ij\in E(G)} (p_j-p_i).    
\end{equation}
Denote that $p(t)$ satisfies the above equation. Then the following inequality holds. 
\begin{equation*}
\frac{d^2}{dt^2}N(p(t))\leq 0,    
\end{equation*}
for any $t\geq 0$. 
\end{theorem}
\begin{proof}
The proof follows from a direct computation. Denote 
\begin{equation*}
\frac{dp_i}{dt}=\frac{1}{2}\sum_{j=1}^n(p_j-p_i)=-\frac{1}{2}L(\Theta)\nabla_pE(p),
\end{equation*}
where $\theta_{ij}=\frac{p_i-p_j}{U'(p_i)-U'(p_j)}$. We next compute the derivatives of $N$ along the equation \eqref{dh}. 
Firstly, 
\begin{equation*}
\begin{split}
\frac{d}{dt}N(p(t))=&e^{-\frac{2}{m}E(p)}\cdot(-\frac{2}{m})\cdot \nabla_pE(p)^{\ts}\frac{dp}{dt}\\
=&\frac{1}{m}e^{-\frac{2}{m}E(p)}\cdot \nabla_pE(p)^{\ts}L(\Theta)\nabla_pE(p)\\
=&\frac{1}{m}e^{-\frac{2}{m}E(p)}\cdot \Gamma_1(p,\nabla_pE(p), \nabla_pE(p)).
\end{split}
\end{equation*}
Secondly, 
\begin{equation*}
\begin{split}
&\frac{d^2}{dt^2}N(p(t))=\frac{d}{dt}(\frac{d}{dt}N(p(t)))\\
=&\frac{1}{m}\Big(\frac{d}{dt}e^{-\frac{2}{m}U(p)}\cdot \Gamma_1(p,\nabla_pE(p), \nabla_pE(p))+e^{-\frac{2}{m}U(p)}\cdot \frac{d}{dt}\Gamma_1(p,\nabla_pE(p), \nabla_pE(p))\Big)\\
=&\frac{1}{m}e^{-\frac{2}{m}U(p)}\Big(\frac{1}{m}\Gamma_1(p,\nabla_pE(p), \nabla_pE(p))^2-\Gamma_2(p,\nabla_pE(p), \nabla_pE(p))\Big)\\
=&\frac{1}{m}e^{-\frac{2}{m}U(p)}\Gamma_1(p,\nabla_pE(p), \nabla_pE(p))^2\Big(\frac{1}{m}-\frac{\Gamma_2(p,\nabla_pE(p), \nabla_pE(p))}{\Gamma_1(p,\nabla_pE(p), \nabla_pE(p))^2}\Big)\\
\leq & 0. 
\end{split}
\end{equation*}
In above derivations, we use the fact in the second equality:
\begin{equation*}
   \frac{d}{dt}\Gamma_1(p,\nabla_pE(p),\nabla_pE(p))=-\Gamma_2(p,\nabla_pE(p),\nabla_pE(p)), 
\end{equation*}
where $p$ solves the discrete heat equation \eqref{dh}. By the definition of constant $m$, we prove the result. 
\end{proof}
We next present the Costa's entropy power's inequality on a two point graph. 
\begin{example}[Two point space]
Consider negative Boltzmann--Shannon entropy in a two point space as
\begin{equation*}
    E(p)=-H(p)=p_1\log p_1+p_2\log p_2.
\end{equation*}
In this case, 
\begin{equation*}
    \begin{split}
    \frac{1}{m}
    =&\min_{p\in M}~\frac{\kappa(p_1, p_2)}{(\log p_1-\log p_2)^2\theta_{12}},
\end{split}
\end{equation*}
where $\theta_{12}=\frac{p_1-p_2}{\log p_1-\log p_2}$ and $\kappa(p_1, p_2)$ is defined in \eqref{lmlocal}. Here $M$ is a line segment. In other words, denote $p_1=x$, $p_2=1-x$, where $x\in [0,1]$. Hence 
\begin{equation*}
\frac{1}{m}= \min_{x\in [0,1]}~\frac{1}{(\log \frac{x}{1-x})(2x-1)}+ \frac{1}{2(\log \frac{x}{1-x})^2x(1-x)}.
\end{equation*}
Numerically, we find that $\frac{1}{m}\approx 1.58353$ at $x\approx 0.058$. 
\end{example}

We remark that our proof is connected but different from the ones in classical Costa's entropy power inequality  \cite{Costa, vil2000}. The major difference comes from the formulation of Gamma calculus in discrete and continuous domain. And the concept of dimension is not clear on a graph, especially for a general function $E$ and a weight function $\theta$. The detailed derivations of Costa's entropy power inequality on general graphs are left in future works.

\section*{Appendix: Hessian matrices of energies along optimal transport dynamics}
In this section, we provide the motivation of this paper. We first present the relation between the metric and the Hessian operator on a finite dimensional metric space. We next review the connection between Gamma operators and second order calculus in optimal transport metric space. In this paper, we formulate these calculations on a discrete spatial domain, such as a finite graph. 
\subsection{Hessian operators in metric spaces}\label{71}
Consider a metric space $(\Omega, g)$. Here $\Omega=\mathbb{R}^n$ and $g\in\mathbb{R}^{n\times n}$ is a smooth positive definite matrix function. We call $g$ the metric function for space $\Omega$. 
Denote a smooth function $E\colon \Omega\rightarrow\mathbb{R}$. The gradient operator of $E$ in $(\Omega, g)$ is defined as
\begin{equation*}
\mathrm{grad}_gE(x)=  g(x)^{-1}\nabla_xE(x)=\Big(\sum_{j=1}^n (g(x)^{-1})_{ij}\nabla_{x_j}E(x)\Big)_{i=1}^n,  
\end{equation*}
where $\nabla_x$ represents the Euclidean gradient operator w.r.t variable $x$. The Hessian operator of $E$ in $(\Omega, g)$ can be written in a tangent space: 
\begin{equation*}
\mathrm{Hess}_gE(x)=\nabla_{x_ix_j}^2E(x)-\sum_{k=1}^d\nabla_{x_k}E(x)\Gamma^k_{ij}(x),
\end{equation*}
where $\Gamma^k_{ij}\colon \Omega\rightarrow\mathbb{R}$ is the Christoffel symbol:
\begin{equation*}
\Gamma^k_{ij}(x)=\frac{1}{2}\sum_{k'=1}^n(g(x)^{-1})_{kk'}\Big(\nabla_{x_i}g_{jk'}(x)+\nabla_{x_j}g_{ik'}(x)-\nabla_{x_{k'}}g_{ij}(x)\Big).
\end{equation*}
The Hessian operator can be formulated on the cotangent space. In other words, denote 
\begin{equation*}
\textrm{Hess}_g^*E(x)=g(x)^{-1}\cdot\textrm{Hess}_gE(x)\cdot g(x)^{-1}.
\end{equation*}
One can also derive the Hessian operator using the geodesic equations. Denote the geodesic equation by 
\begin{equation*}
\frac{d^2}{dt^2} x_k+\sum_{i=1}^n\sum_{j=1}^n\Gamma_{ij}^k\frac{dx_i}{dt}\frac{dx_j}{dt}=0.
\end{equation*}
Denote a vector $f(t)\in\mathbb{R}^n$ such that $\frac{dx}{dt}=g(x)^{-1}f$. Then the geodesic equation in term of $(x(t),f(t))$ satisfies 
\begin{equation*}
\left\{\begin{aligned}
\frac{dx}{dt}=& g(x)^{-1}f,\\
\frac{df}{dt}=&-\frac{1}{2}\nabla_x(f^{\ts}g(x)^{-1}f).
\end{aligned}\right.
\end{equation*}
In this case, the Hessian operator of $E$ in $(\Omega, g)$ satisfies
\begin{equation*}
\frac{d^2}{dt^2}E(x(t))=f(t)^\ts\mathrm{Hess}_g^*E(x(t))f(t),    
\end{equation*}
where $(x(t), f(t))$ satisfies the geodesic equation. 

We next apply the Hessian operator to study the convergence behavior of gradient flows. Consider 
\begin{equation*}
\frac{d}{dt}x(t)=- g(x(t))^{-1}\nabla_xE(x(t)).
\end{equation*}
It is a gradient flow equation of a convex function $E$ in $(\Omega, g)$. We use the Lyaponuv function to study above gradient flows. Along the gradient flow, we have the following estimates. Firstly, 
\begin{equation*}
\frac{d}{dt}E(x(t))=-(\nabla_x E(x), g(x)^{-1}\nabla_xE(x)).
\end{equation*}
Secondly, 
\begin{equation*}
    \frac{d^2}{dt^2}E(x(t))=2(\nabla_xE(x), \mathrm{Hess}_g^*E(x)\nabla_xE(x)). 
\end{equation*}
If there exists a constant $\kappa>0$, such that         $\mathrm{Hess}_gE(x)\succeq \kappa g(x)$, i.e.,
\begin{equation*}
    \mathrm{Hess}^*_gE(x)\succeq \kappa g(x)^{-1},
\end{equation*}
then we have 
\begin{equation*}
\frac{d^2}{dt^2}E(x(t))\geq -2\kappa\frac{d}{dt}E(x(t)).
\end{equation*}
Integrating in a time domain $[t, \infty)$, we have the functional inequality 
\begin{equation*}
E(x)-E(x_*)\leq \frac{1}{2\kappa} (\nabla_xE(x), g(x)^{-1}\nabla_xE(x)),    
\end{equation*}
where $x_*=\arg\min_{x\in\Omega} E(x)$. This also means 
\begin{equation*}
\frac{d}{dt}(E(x(t))-E(x_*))\leq -2\kappa (E(x(t))-E(x_*)).    
\end{equation*}
Following the Grownwall's inequaity, we have 
\begin{equation*}
E(x(t))-E(x_*)\leq e^{-2\kappa t}(E(x_0)-E(x_*)),    
\end{equation*}
where $x_0$ is the initial condition for the gradient flow. 

The above formulations are major motivations of this paper. We shall focus on a graph dependent metric function $g$, and derive the Hessian matrix for general function $E$. This is the proposed mean-field information Gamma calculus. And the constant $\kappa$ is the proposed mean-field Ricci curvature lower bound. 
\subsection{Hessian operators in optimal transport spaces}
We next present the Wasserstein-$2$ metric (optimal transport metric) and demonstrate its Hessian operator of a functional; see details in \cite{vil2008}. 

Let $(\Omega, g)=(\mathbb{T}^d, \mathbb{I})$ be a $d$ dimensional torus, where $\mathbb{I}\in \mathbb{R}^{d\times d}$ is an identity matrix and $(\cdot, \cdot)$ denotes the Euclidean inner product. 
Let $\int$ be the integration over domain $\Omega$. Denote a smooth positive probability density space as
\begin{equation*}
\mathcal{P}=\Big\{\rho\in C^{\infty}(\Omega)\colon \int \rho dx=1,\quad \rho> 0\Big\}.
\end{equation*}
The tangent space at $\rho \in \mathcal{P}$ is given as
\begin{equation*}
T_\rho \mathcal{P}=\Big\{\sigma\in C^\infty(\Omega)\colon \int \sigma dx=0\Big\}.
\end{equation*}
Define a weighted Laplacian operator as
\begin{equation*}
\Delta_a=\nabla\cdot(a\nabla), 
\end{equation*}
where $a\in C^{\infty}(M)$ is a given smooth function. In other words, 
\begin{equation*}
\int (f_1, \Delta_a f_2) dx= -\int (\nabla f_1, \nabla f_2)a dx,
\end{equation*}
for any functions $f_1$, $f_2\in C^\infty(\Omega)$. The Wasserstein-2 metric is defined below. 
\begin{definition}[Optimal transport metric]\label{metric}
  The inner product $g\colon \mathcal{P}\times {T_\rho}\mathcal{P}\times{T_\rho}\mathcal{P}\rightarrow\mathbb{R}$ is defined by  
  \begin{equation*} 
 g(\rho)(\sigma_1, \sigma_2)=\int \big(\sigma_1, (-\Delta_\rho)^{-1}\sigma_2\big)=\int (\nabla f_1, \nabla f_2)\rho dx,
  \end{equation*}
  for any $\sigma_1,\sigma_2\in T_\rho\mathcal{P}$. Here $\Delta_\rho=-\nabla\cdot(\rho\nabla)$ is an elliptic operator weighted linearly in density function $\rho$, and $f_i=-\Delta_\rho\sigma_i$, $i=1,2$.
  \end{definition}

We next present the Hessian operator in Wasserstein-2 metric space. 
Given a smooth functional $\mathcal{E}\colon \mathcal{P}\rightarrow \mathbb{R}$, the Hessian operator of $\mathcal{E}$ in $(\mathcal{P}, g)$ satisfies
\begin{equation*}
\begin{split}
\textrm{Hess}_g\mathcal{E}(\rho)(\sigma_1, \sigma_2)=&\int \int \nabla_x\nabla_y\delta^2\mathcal{E}(\rho)(x,y)(\nabla f_1(x),\nabla f_2(y))\rho(x)\rho(y)dxdy\\
&+\int\nabla_{xx}^2\delta\mathcal{E}(\rho)(\nabla f_1(x), \nabla f_2(x))\rho(x)dx,
\end{split}
\end{equation*}
where $\sigma_i=-\nabla\cdot(\rho\nabla f_i)$, $i=1,2$, and $\delta$, $\delta^2$ are the first and the second $L^2$ variation operators, respectively. 

\begin{example}[Linear energy]\label{linear}
Consider 
\begin{equation*}
\mathcal{E}(\rho)=\int V(x)\rho(x) dx,
\end{equation*}
where $V\in C^{2}(\Omega)$ is a second order differentiable function. Then 
\begin{equation*}
\mathrm{Hess}_g\mathcal{E}(\rho)(\sigma_1, \sigma_2)=\int \nabla_{xx}^2V(x)(\nabla f_1(x), \nabla f_2(x))\rho(x)dx.
\end{equation*}
\end{example}

\begin{example}[Interaction energy]\label{inter}
Consider 
\begin{equation*}
\mathcal{E}(\rho)=\int \int W(x,y)\rho(x)\rho(y)dxdy.
\end{equation*}
where $W(x,y)=W(y,x)\in C^{2}(\Omega\times\Omega)$ is a given kernel function. Then 
\begin{equation*}
\begin{split}
\mathrm{Hess}_g\mathcal{E}(\rho)(\sigma_1, \sigma_2)=&\int \int \nabla_x\nabla_y W(x,y) (\nabla_x f_1(x),\nabla_yf_2(y))\rho(x)\rho(y)dxdy\\
&+\int \nabla_{xx}^2W(x,y)(\nabla f_1(x), \nabla f_2(x))\rho(x)dx.
\end{split}
\end{equation*}
\end{example}

\begin{example}[Negative Boltzmann--Shannon Entropy]\label{ent}
Consider 
\begin{equation*}
\mathcal{E}(\rho)=-\mathcal{H}(\rho)=\int \rho(x)\log\rho(x) dx.
\end{equation*}
Then
\begin{equation*}
\mathrm{Hess}_g\mathcal{E}(\rho)(\sigma_1, \sigma_2)=\int \mathrm{tr}(\nabla^2 f_1(x):\nabla^2f_2(x)) \rho(x) dx.
\end{equation*}
\end{example}
In above examples, when $\mathcal{E}$ is chosen as the negative Boltzmann-Shannon entropy in Example \ref{ent}, then $(\nabla f_1, \nabla f_2)$, $\mathrm{tr}(\nabla^2 f_1\colon \nabla^2 f_2)$ are known as Gamma one, Gamma two operators, respectively. In other words, the optimal transport metric is the integration of Gamma one operator w.r.t density $\rho$, while the Hessian operator of entropy is the integration of Gamma two operator w.r.t density $\rho$. We remark that Examples \ref{linear}, \ref{inter}, \ref{ent} are analogs of the bi-linear forms in Corollary \ref{Col1}, where $\theta$ is chosen as a homogeneous of degree one function. By generalizing these facts on graphs, we formulate the proposed mean-field information matrices. 

\begin{thebibliography}{10}

  \bibitem{AGS} 
 L.~Ambrosio, N.~Gigli and G.~Savare.
\newblock{\em Gradient Flows in Metric Spaces and in the Space of Probability Measures}, 2008.
 
 \bibitem{BCGL}
J.~Backhoff, G.~Conforti, I.~Gentil, and C.~Leonard. 
\newblock{The mean field Schrödinger problem: ergodic behavior, entropy estimates and functional inequalities.}
\newblock{\em Probab. Theory Relat. Fields}, 178, 475–530, 2020. 

\bibitem{BE}
D.~Bakry and M.~{\'E}mery.
\newblock {Diffusions hypercontractives}.
\newblock {\em S{\'e}minaire de probabilit{\'e}s de Strasbourg}, 19:177--206, 1985.

 \bibitem{BGL} 
 D.~Bakry, I.~Gentil and M.~Ledoux. 
 \newblock{\em Analysis and geometry of Markov diffusion operators.} 
 \newblock{Springer}, 2014.
 
\bibitem{BF}
F.~Baudoin. 
\newblock{Bakry-Emery meets Villani.}
\newblock{\em J. Funct. Anal.} 273, 2275–2291, 2017.

\bibitem{BT}  
S.~G.~Bobkov and P.~Tetali. 
\newblock{Modified Logarithmic Sobolev Inequalities in Discrete Settings.}
\newblock{\em J Theor Probab}, 19:289–336, 2006.   
  
\bibitem{GC}
G.~Conforti.
\newblock{A probabilistic approach to convex ($\phi$)-entropy decay for Markov chains.}
\newblock{\em arXiv:2004.10850}, 2020.

\bibitem{chow2012}
S.N.~Chow, W.~Huang, Y.~Li and H.~Zhou.
\newblock Fokker--Planck equations for a free energy functional or Markov process on a graph.
\newblock {\em Archive for Rational Mechanics and Analysis}, 203(3):969--1008,  2012.

\bibitem{Li-entropy}
{S.N.~Chow, W.~Li and H.~Zhou.}
 \newblock Entropy dissipation of Fokker-Planck equations on finite graphs.
\newblock{\em Discrete and Continuous Dynamical Systems-A}, 2018.
 

\bibitem{Costa}
M.H.M.~Costa. 
\newblock{A new entropy power inequality.}
\newblock{IEEE Trans. on Information Theory}, vol. 31, no. 6, pp. 751-760, 1985. 

\bibitem{Csiszar}
I.~Csisz\'{a}r and P.~C. Shields.
\newblock Information theory and statistics: A tutorial.
\newblock {\em Commun. Inf. Theory}, 1(4):417--528, Dec. 2004.

\bibitem{Duncan}
A.B. Duncan, G.A. Pavliotis, and K.C. Zygalakis.
\newblock{Nonreversible Langevin Samplers: Splitting Schemes, Analysis and Implementation.}
\newblock{arXiv: 1701.04247}, 2017.

\bibitem{EMR2}
M.~Erbar and M.~Fathi.
\newblock Poincar{\'e}, modified logarithmic {{Sobolev}} and isoperimetric
  inequalities for {{Markov}} chains with non-negative {{Ricci}} curvature.
\newblock {\em Journal of Functional Analysis}, 274(11):3056--3089, 2018.

\bibitem{EMR3}
M.~Erbar, C.~Henderson, G.~Menz, and P.~Tetali.
\newblock Ricci curvature bounds for weakly interacting {{Markov}} chains.
\newblock {\em Electronic Journal of Probability}, 22, 2017.

\bibitem{Maas2012}
M.~Erbar and J.~Maas.
\newblock Ricci {{Curvature}} of {{Finite Markov Chains}} via {{Convexity}} of
  the {{Entropy}}.
\newblock {\em Archive for Rational Mechanics and Analysis}, 206(3):997--1038,
  2012.

\bibitem{EMR4}
M.~Erbar, J.~Maas, and P.~Tetali.
\newblock Discrete {{Ricci Curvature}} bounds for {{Bernoulli}}-{{Laplace}} and
  {{Random Transposition}} models.
\newblock {\em Annales de la facult{\'e} des sciences de Toulouse
  Math{\'e}matiques}, 24(4):781--800, 2015.

\bibitem{EMR1}
M.~Fathi and J.~Maas.
\newblock Entropic {{Ricci}} curvature bounds for discrete interacting systems.
\newblock {\em The Annals of Applied Probability}, 26(3):1774--1806, 2016.

\bibitem{GLM}
W.~Gangbo, W.~Li, and C.~Mou.
\newblock{ Geodesic of minimal length in the set of probability measures on graphs.}
\newblock{\em ESAIM: COCV}, 2019. 

\bibitem{Gross}
L. Gross. 
\newblock Logarithmic Sobolev inequalities.
\newblock {\em American Journal of Mathematics,} 97(4), 1061--1083, 1975.

 \bibitem{GSV}
 D. Guo, S. Shamai and S. Verdu. 
 \newblock{Proof of Entropy Power Inequalities Via MMSE.} \newblock{\em IEEE International Symposium on Information Theory}, pp. 1011-1015, 2006. 

\bibitem{Jost2}
B.~Hua, J.~Jost, and S.~Liu.
\newblock Geometric {{Analysis Aspects}} of {{Infinite Semiplanar Graphs}} with
  {{Nonnegative Curvature}}.
\newblock {\em Journal f{\"u}r die reine und angewandte Mathematik (Crelles
  Journal)}, 2015(700):1--36, 2015.

\bibitem{Jost}
J.~Jost and S.~Liu.
\newblock Ollivier's {{Ricci Curvature}}, {{Local Clustering}} and
  {{Curvature}}-{{Dimension Inequalities}} on {{Graphs}}.
\newblock {\em Discrete Comput. Geom.}, 51(2):300--322, 2014.

\bibitem{Lafferty}
J.~D. Lafferty.
\newblock The {{Density Manifold}} and {{Configuration Space Quantization}}.
\newblock {\em Transactions of the American Mathematical Society},
  305(2):699--741, 1988.

\bibitem{LL}
F.~Leger and W.~Li.
\newblock{Hopf--Cole transformation via generalized Schr{\"o}dinger bridge problem.}
\newblock{\em Journal of Differential Equations}, Volume 274, 788--827, 2021. 

\bibitem{LiHess}
W.~Li.
\newblock Hessian metric via transport information geometry. 
\newblock{\em Journal of Mathematical Physics}, 62, 033301, 2021. 

\bibitem{LiG}
W.~Li.
\newblock Transport information geometry: Riemannian calculus on graphs.
\newblock {\em Information Geometry}, 2021.

\bibitem{LiG1}
W.~Li.
\newblock Diffusion Hypercontractivity via Generalized Density Manifold. 
\newblock {\em arXiv:1907.12546}, 2019.

\bibitem{LiR}
W.~Li, G.~Montúfar. 
\newblock{Ricci curvature for parametric statistics via optimal transport.}
\newblock{Information Geometry}, 3, 89–117, 2020. 

\bibitem{LinYau2}
Y.~Lin, L.~Lu, and S.T.~Yau.
\newblock Ricci {{Curvature}} of {{Graphs}}.
\newblock {\em Tohoku Mathematical Journal}, 63(4):605--627, 2011.

\bibitem{LinYau}
Y.~Lin and S.T.~Yau.
\newblock Ricci {{Curvature}} and {{Eigenvalue Estimate}} on {{Locally Finite
  Graphs}}.
\newblock {\em Mathematical Research Letters}, 17(2):343--356, 2010.

 \bibitem{LLPS}
 R. Liu, T. Liu, H.V. Poor and S. Shamai.
 \newblock{A Vector Generalization of Costa's Entropy-Power Inequality With Applications.}
 \newblock{\em IEEE Transactions on Information Theory}, vol. 56, no. 4, pp. 1865-1879, 2010. 



\bibitem{Lott_Villani}
J.~Lott and C.~Villani.
\newblock Ricci {{Curvature}} for {{Metric}}-{{Measure Spaces}} via {{Optimal
  Transport}}.
\newblock {\em Annals of Mathematics}, 169(3):903--991, 2009.

\bibitem{EM1}
J.~Maas.
\newblock{Gradient flows of the entropy for finite Markov chains.}
\newblock{\em Journal of Functional Analysis}, 261(8) 2250--2292, 2011.

\bibitem{M}
A.~Mielke.
\newblock A gradient structure for reaction--diffusion systems and for energy-drift-diffusion.
\newblock{\em Nonlinearity}, 24(4)13-29 2011.

\bibitem{M1}
A.~Mielke. 
\newblock Geodesic convexity of the relative entropy in reversible Markov chains. 
\newblock{\em Calculus of Variations and Partial Differential Equations}, 48(1-2):1–31, 2013. 

\bibitem{Ollivier_Ricci}
Y.~Ollivier.
\newblock Ricci {{Curvature}} of {{Markov Chains}} on {{Metric Spaces}}.
\newblock {\em Journal of Functional Analysis}, 256(3):810--864, 2009.

\bibitem{OV}
F.~Otto and C.~Villani.
\newblock Generalization of an {{Inequality}} by {{Talagrand}} and {{Links}}
  with the {{Logarithmic Sobolev Inequality}}.
\newblock {\em Journal of Functional Analysis}, 173(2):361--400, 2000.

\bibitem{strum}
K.T. Sturm.
\newblock On the {{Geometry}} of {{Metric Measure Spaces}}.
\newblock {\em Acta Mathematica}, 196(1):65--131, 2006.

  \bibitem{vil2000}
  C.~Villani. 
\newblock{Concavity of entropy power.}
\newblock{\em IEEE Trans. Info. Theory}, 46 (4), 1695-1696, 2000. 

\bibitem{vil2008}
C.~Villani.
\newblock {\em Optimal {{Transport}}: {{Old}} and {{New}}}.
\newblock Number 338 in Grundlehren der mathematischen Wissenschaften.
  {Springer}, Berlin, 2009.
 
 \bibitem{WZ} 
F.~Weber, R.~Zacher.
\newblock{ The entropy method under curvature-dimension conditions in the spirit of Bakry-Émery in the discrete setting of Markov chains.}
\newblock{\em Journal of Functional Analysis}, Volume 281, (5), 0022-1236, 2021. 
\end{thebibliography}
\end{document}